\documentclass[12pt,reqno]{amsart}
\usepackage[a4paper]{geometry}
\usepackage{tikz}
\usepackage{amssymb,amsmath,amsthm,enumerate}

\usepackage{soul} %temporary, to strikethrough a text

\DeclareMathOperator{\Tr}{Tr}
\DeclareMathOperator{\Ker}{Ker}
\DeclareMathOperator{\Ran}{Ran}
\DeclareMathOperator{\Dom}{Dom}

\DeclareMathOperator{\sign}{sign}
\DeclareMathOperator{\esssup}{ess\, sup}

\renewcommand{\Re}{\operatorname{Re}}
\renewcommand{\Im}{\operatorname{Im}}

\newcommand{\abs}[1]{\lvert#1\rvert}
\newcommand{\norm}[1]{\lVert#1\rVert}
\newcommand{\jap}[1]{\langle#1\rangle}

\newcommand{\bbN}{{\mathbb N}}
\newcommand{\bbR}{{\mathbb R}}
\newcommand{\bbC}{{\mathbb C}}

\newcommand{\bH}{\mathbf{H}}
\newcommand{\bR}{\mathbf{R}}
\newcommand{\bJ}{\mathbf{J}}

\newcommand{\calM}{\mathcal{M}}
\newcommand{\calC}{\mathcal{C}}
\newcommand{\calD}{\mathcal{D}}
\newcommand{\calO}{\mathcal{O}}

\newcommand{\dd}{\mathrm d}
\newcommand{\ii}{\mathrm i}
\newcommand{\ee}{\mathrm e}
\newcommand{\loc}{\mathrm{loc}}

\numberwithin{equation}{section}

\theoremstyle{plain}
\newtheorem{theorem}{\bf Theorem}[section]
\newtheorem*{theorem*}{Theorem}
\newtheorem{lemma}[theorem]{\bf Lemma}
\newtheorem{proposition}[theorem]{\bf Proposition}
\newtheorem*{proposition*}{\bf Proposition}

\theoremstyle{definition}
\newtheorem{definition}[theorem]{\bf Definition}
\newtheorem*{definition*}{\bf Definition}

\theoremstyle{remark}
\newtheorem*{remark*}{\bf Remark}
\newtheorem{remark}[theorem]{\bf Remark}

\newcommand{\wt}{\widetilde}
\newcommand{\eps}{\varepsilon}

\newcommand\smallO{ %to use small curvy Landau o
  \mathchoice
    {{\scriptstyle\mathcal{O}}}% \displaystyle
    {{\scriptstyle\mathcal{O}}}% \textstyle
    {{\scriptscriptstyle\mathcal{O}}}% \scriptstyle
    {\scalebox{.7}{$\scriptscriptstyle\mathcal{O}$}}%\scriptscriptstyle
  }
  
\begin{document}

\title[The Borg--Marchenko theorem for complex potentials]{The Borg-Marchenko uniqueness theorem for complex potentials}

\author{Alexander Pushnitski}
\address{Department of Mathematics, King's College London, Strand, London, WC2R~2LS, United Kingdom}
\email{alexander.pushnitski@kcl.ac.uk}

\author{Franti\v{s}ek \v{S}tampach}
\address{Department of Mathematics, Faculty of Nuclear Sciences and Physical Engineering, Czech Technical University in Prague, Trojanova 13, 12000 Prague~2, Czech Republic.}
\email{frantisek.stampach@cvut.cz}

\subjclass[2020]{34L40, 37K15}

\keywords{non-self-adjoint Schr{\" o}dinger operator, complex potential, spectral pair, spectral measure, inverse spectral problem, Borg--Marchenko theorem}

\date{\today}

\begin{abstract}
We introduce and study a new theoretical concept of \emph{spectral pair} for a Schr{\" o}dinger operator $H$ in $L^2(\bbR_+)$ with a bounded \emph{complex-valued} potential. The spectral pair consists of a scalar measure and a complex-valued function. We show that in many ways, the spectral pair generalises the classical spectral measure to the non-self-adjoint case. First, extending the classical Borg-Marchenko theorem, we prove a uniqueness result: the spectral pair uniquely determines the operator $H$. Second, we derive asymptotic formulas for the spectral pair in the spirit of the classical result of Marchenko. In the case of real-valued potentials, we relate the spectral pair to the spectral measure of $H$. Lastly, we provide formulas for the spectral pair at a~simple eigenvalue of~$|H|$.
\end{abstract}

\maketitle

%%%%%%%%%%%%%%%%%%%%%%%%%%%%%%
%%%%%%%%%%%%%%%%%%%%%%%%%%%%%%
\section{Introduction and main results} %\label{sec.a}
%%%%%%%%%%%%%%%%%%%%%%%%%%%%%%
%%%%%%%%%%%%%%%%%%%%%%%%%%%%%

\subsection{Foreword}
In recent years, various problems in operator and spectral theory (eigenvalue and eigenfunction estimates, spectral asymptotics, semiclassical analysis, etc.) that were traditionally formulated for classes of self-adjoint operators have become topics of renewed interest when the assumption of self-adjointness is lifted. For non-self-adjoint operators, powerful tools such as the spectral theorem are missing, and one often has to resort to using a range of \emph{ad hoc} techniques~\cite{Dav,Hel,TE}.  In the present paper, we introduce the notion of \emph{spectral pair} for the class of non-self-adjoint Schr{\" o}dinger operators on the half-line with bounded complex-valued potentials. We demonstrate that the spectral pair can serve as a~natural substitute for the spectral measure. 

\subsection{Warm-up: the classical self-adjoint case}\label{sec.warmup}
As a warm-up, let us recall basic statements of the Titchmarsh--Weyl spectral theory and the Borg--Marchenko uniqueness theorem in the notation adapted for our purposes. We denote $\bbR_+=(0,\infty)$. Let $q$ be a bounded (for simplicity) real-valued measurable function on $\bbR_+$, and let $H$ be the Schr{\"o}dinger operator
\begin{equation}
Hf=-f''+qf
\label{eq:def_H}
\end{equation}
on $L^2(\bbR_+)$ with the Robin boundary condition 
\begin{equation}
f'(0)+\alpha f(0)=0,
\label{bc0}
\end{equation}
where $\alpha\in\mathbb{R}\cup\{\infty\}$. The case $\alpha=\infty$ corresponds to the Dirichlet boundary condition $f(0)=0$. Then $H$ is self-adjoint with the domain 
\begin{equation}
\Dom H=\{f\in W^{2,2}([0,\infty)) \mid \text{ $f$ satisfies \eqref{bc0}} \},
\label{eq:dom_H}
\end{equation}
where $W^{2,2}([0,\infty))$ is the standard Sobolev space. 
It is sometimes advantageous to rewrite the boundary condition~\eqref{bc0} as
\begin{equation}
f(0)\cos\gamma+f'(0)\sin\gamma=0,
\label{bc1}
\end{equation}
with a new boundary parameter $\gamma\in[0,\pi)$ determined by equations
\begin{equation}
\sin\gamma=\frac{1}{\sqrt{1+\alpha^{2}}},
\quad
\cos\gamma=\frac{\alpha}{\sqrt{1+\alpha^{2}}}.
\label{eq:def_sc_scalar}
\end{equation}
The cases $\gamma=0$ (Dirichlet) and $\gamma=\pi/2$ (Neumann) correspond to $\alpha=\infty$ and $\alpha=0$, respectively.

For $\lambda\in\bbC$, we denote by $\varphi$ and $\theta$ the scalar solutions (the fundamental system) to the eigenvalue equation
\[
-f''+qf=\lambda f
\]
with the Cauchy data
\begin{equation}
\begin{aligned}
\varphi(0,\lambda)&=\sin\gamma,\\
 \varphi'(0,\lambda)&=-\cos\gamma,
\end{aligned}
\qquad
\begin{aligned}
\theta(0,\lambda)=\cos\gamma, \\
\theta'(0,\lambda)=\sin\gamma.
\end{aligned}
\label{eq:X11}
\end{equation}
For $\lambda\in\bbC\setminus\bbR$, the $m$-function $m_\alpha(\lambda)$ is the unique complex number such that the solution 
\[
\chi(x,\lambda)=\theta(x,\lambda)-\varphi(x,\lambda) m_\alpha(\lambda)
\]
belongs to $L^2(\bbR_+)$ as a function of $x$. The $m$-function is a Herglotz--Nevanlinna function, i.e. 
it is analytic in the upper and lower half-planes and satisfies $\overline{m_\alpha(\lambda)}=m_\alpha(\overline{\lambda})$ and 
$\Im m_\alpha (\lambda)>0$ for $\Im \lambda>0$. As a consequence,  it can be represented in the form
\begin{equation}
m_\alpha(\lambda)=
\Re m_\alpha(\ii)+\int_{-\infty}^\infty \left(\frac{1}{t-\lambda}-\frac{t}{1+t^2}\right)\dd\sigma(t), 
\quad \Im \lambda\neq0,
\label{eq.r4}
\end{equation}
where $\sigma$ is a positive scalar measure called \emph{the spectral measure of $H$}. 
(For a general Herglotz--Nevanlinna function, the representation~\eqref{eq.r4} may contain a linear in $\lambda$ term, but in this case it does not appear.) The term ``spectral measure'' is motivated by the fact that the operator $H$ is unitarily equivalent to the operator of multiplication by the independent variable in $L^2_\sigma(\bbR)$, i.e. the space of all functions on $\bbR$ square integrable with respect to the measure $\sigma$. In particular, the support of $\sigma$ coincides with the spectrum of $H$. 

The classical Borg--Marchenko uniqueness theorem \cite{Borg,M} says that the \emph{spectral map} 
\[
(q,\alpha)\mapsto \sigma
\]
is injective, i.e. the spectral measure $\sigma$ uniquely determines both the boundary parameter $\alpha$ and the potential $q$. More recent works with alternative proofs and further developments include~\cite{S1} and~\cite{B}.

The description of the image of the spectral map is a much more delicate question. In the classical work of Gelfand and Levitan \cite{Gelfand-Levitan51}, necessary and sufficient conditions were given for a measure $\sigma$ to ensure that it corresponds to some $n$ times continuously differentiable potential $q$. There was a small gap between necessary and sufficient conditions, which was subsequently closed by M.~G.~Krein in \cite{Krein2}. A complete proof can be found, for example, in Chapter 2 of the monograph \cite{Levitan}. For the general class of potentials $q\in L_{\loc}^{1}(\bbR_{+})$, which are limit-point at infinity, a description of the image of the spectral map was obtained by Remling in~\cite{R} using the concept of de Branges spaces. In a~more recent work \cite{KillipSimon}, the spectral measures $\sigma$ corresponding to $q\in L^2(\bbR_+)$ have been characterised.

\subsection{The operator $H$ with complex potential}
The main object of this paper is the Schr{\"o}dinger operator 
\begin{equation}
Hf=-f''+qf \quad \text{ on $L^2(\bbR_+)$}
\label{eq.r4a}
\end{equation}
with a \emph{complex-valued} potential $q$ and a \emph{non-self-adjoint} Robin boundary condition 
\begin{equation}
f'(0)+\alpha f(0)=0,
\label{eq.r4b}
\end{equation}
where $\alpha\in\bbC\cup\{\infty\}$. 
In order to keep our exposition as simple as possible, we will assume throughout the paper that the potential $q$ is \emph{bounded}. The operator $H$ defined by \eqref{eq.r4a} with the domain 
\[
\Dom H=\{f\in W^{2,2}([0,\infty)) \mid \text{ $f$ satisfies \eqref{eq.r4b}} \}
\]
is closed in $L^2(\bbR_+)$. Since both $q$ and $\alpha$ can be complex-valued, the operator $H$ is in general non-self-adjoint. The purpose of this paper is to suggest a new concept of \emph{spectral pair} for $H$ (see Definition~\ref{def:spectral_data}) which should be considered as a substitute for the spectral measure. We will prove a Borg--Marchenko type uniqueness theorem for the spectral pair and establish simple properties of the corresponding spectral map.

\subsection{The operator $\bH$}
Our approach is to access the spectral properties of the non-self-adjoint operator $H$ by using the language of self-adjoint spectral theory. To this end, we consider the so-called hermitisation of $H$, i.e. the \emph{self-adjoint} block-matrix operator
\[
\bH=
\begin{pmatrix}
0& H\\ H^*& 0
\end{pmatrix}
\quad\text{ in } L^2(\bbR_+)\oplus L^2(\bbR_+).
\]
We note that the adjoint operator $H^*$ corresponds to the complex-conjugate potential $\overline{q}$ and the boundary condition 
\[
f'(0)+\overline{\alpha}f(0)=0,
\]
with the complex-conjugate boundary parameter $\alpha$ (compare to \eqref{eq.r4b}). 

By identifying $L^2(\bbR_+)\oplus L^2(\bbR_+)$ with the space of $\bbC^2$-valued functions $L^2(\bbR_+;\bbC^2)$, we can view $\bH$ as operator
\begin{equation}
\quad 
\bH=-\epsilon\frac{\dd^2}{\dd x^2}+Q, 
\quad\mbox{ with }\;
\epsilon:=
\begin{pmatrix}
0&1\\1&0
\end{pmatrix},
\label{a0}
\end{equation}
acting in $L^2(\bbR_+;\bbC^2)$, 
where $Q$ is a $2\times 2$ Hermitian matrix-valued function of the special form
\begin{equation}
Q=\begin{pmatrix}0&q\\ \overline{q}&0\end{pmatrix}.
\label{eq:a2}
\end{equation}
The operator $\bH$ is supplied with the boundary condition 
\begin{equation}
F'(0)+AF(0)=0,
\label{eq:a2a0}
\end{equation}
where
\begin{equation}
A=\begin{pmatrix}
\overline{\alpha} & 0 \\ 0 & \alpha
\end{pmatrix}.
\label{eq:def_A}
\end{equation}
If $\alpha=\infty$, \eqref{eq:a2a0} is to be interpreted as the Dirichlet boundary condition $F(0)=0$.
With this definition, $\bH$ is a self-adjoint operator in $L^2(\bbR_+;\bbC^2)$ with the domain 
\begin{equation}
 \Dom\bH=\{F\in W^{2,2}([0,\infty);\bbC^{2}) \mid \text{ $F$ satisfies \eqref{eq:a2a0}} \}.
\label{eq:dom_bH}
\end{equation}

Similarly to~\eqref{bc1}, the boundary condition~\eqref{eq:a2a0} can be equivalently written as
\begin{equation}
CF(0)+S\epsilon F'(0)=0,
\label{eq:a2a}
\end{equation}
where
\begin{equation}
\begin{aligned}
S & =\frac{1}{\sqrt{1+|\alpha|^{2}}}\,I,
&
C & =\frac{1}{\sqrt{1+|\alpha|^{2}}}\,\epsilon A, && \text{ if } \alpha\in\bbC,
\\
S & =0, &  C & =\epsilon, && \text{ if } \alpha=\infty.
\end{aligned}
\label{eq:a2b}
\end{equation}

The matrices $S$ and $C$ are self-adjoint and satisfy the identities
\begin{equation}
SC-CS=0, \quad S^{2}+C^{2}=I.
\label{eq:S-C_id}
\end{equation}

We prefer to start by considering the operator $\bH$ with a general bounded $2\times2$ Hermitian matrix potential $Q$, and after that to specify to $Q$ of the particular form \eqref{eq:a2}. 
\begin{remark*}
One can in principle consider the operator $\bH$ with a more general matrix $A$ in the boundary condition~\eqref{eq:a2a0}, but for our purposes it is sufficient to confine ourselves to $A$ of the form \eqref{eq:def_A}. 
\end{remark*}

%%%%%%%%%%%%%%%%%%%%%%%
\subsection{Main results}
%%%%%%%%%%%%%%%%%%%%%%%

For the operator $\bH$ one defines a $2\times 2$ matrix-valued spectral measure which we denote by $\Sigma$. The definition of $\Sigma$ is standard and completely analogous to the scalar case;  we postpone the details to Section~\ref{sec:a3}. We have the following uniqueness theorem for \emph{general} $2\times 2$ Hermitian matrix-valued potentials~$Q$. 

%%%%%%%%%%%%%%%%%%
\begin{theorem}\label{thm:a1}
%%%%%%%%%%%%%%%%%%
Let $\bH$ be the self-adjoint operator defined by \eqref{a0} with a~bounded measurable $2\times 2$ Hermitian matrix-valued potential $Q$ and the boundary condition \eqref{eq:a2a0}, \eqref{eq:def_A} with $\alpha\in\bbC\cup\{\infty\}$. Let $\Sigma$ be the corresponding $2\times 2$ matrix-valued spectral measure. Then the map 
\[
(Q,\alpha)\mapsto \Sigma
\] 
is injective, i.e. both the potential $Q$ and the boundary parameter $\alpha$ are uniquely determined by the spectral measure $\Sigma$. 
\end{theorem}

At this point we would love to say that this theorem is known and give a reference. Unfortunately, we were not able to find it in the literature. All uniqueness results of this genre, that we are aware of, assume that $\epsilon$ in \eqref{a0} is the identity matrix (which could be easily modified to positive definite matrices). We emphasize that our $\epsilon$ is sign indefinite, which makes the spectral theory of $\bH$ non-standard; in fact, in some ways $\bH$ is closer to a Dirac operator than to a Schr\"odinger operator. In any case, we give a proof of Theorem~\ref{thm:a1} in Sections~\ref{sec.d} and \ref{sec.c}.

Next, we specialise to the case of $Q$  given by \eqref{eq:a2}. In Section~\ref{sec.b3a} we will prove the following statement.
%%%%%%%%%%%%%%%%%%
\begin{theorem}\label{thm:a2}
%%%%%%%%%%%%%%%%%%
Let $Q$ be as in \eqref{eq:a2} with a bounded complex-valued measurable potential~$q$. Then there exists a unique \underline{even} positive measure $\nu$ on $\bbR$ and a unique \underline{odd} complex-valued function $\psi\in L^\infty(\nu)$ satisfying $\abs{\psi(s)}\leq1$ for $\nu$-a.e. $s\in\bbR$ such that
\begin{equation}
\boxed{
\dd\Sigma=
\begin{pmatrix}
1&\psi
\\
\overline{\psi}&1
\end{pmatrix}\dd\nu.
}
\label{eq:a1}
\end{equation}
\end{theorem}

Observe that $\psi$ is only well defined on the support of $\nu$. Theorem~\ref{thm:a2} states that the spectral measure of an operator $\bH$ with the potential $Q$ of the special structure~\eqref{eq:a2}, arising by the hermitisation of $H$, is naturally parameterised by the two parameters $\nu$ and $\psi$. We observe that since $\psi$ is odd, we necessarily have $\psi(0)=0$ if $\nu$ has a point mass at zero (otherwise $\psi(0)$ is not defined). Since $\nu$ is even and $\psi$ is odd, both are determined by their restrictions onto $[0,\infty)$. It is often more convenient to think of thus restricted pair $(\nu,\psi)$.

\begin{definition}\label{def:spectral_data}
We call the pair $(\nu,\psi)$, defined by Theorem~\ref{thm:a2},
the \emph{spectral pair} of the Schr\"odinger operator $H$. 
\end{definition}

We are interested in the properties of the \emph{spectral map}
\[
(q,\alpha)\mapsto (\nu,\psi).
\]
As a consequence of Theorems~\ref{thm:a1} and \ref{thm:a2} we have the injectivity of this map, which is the main result of this paper. 
%%%%%%%%%%%%%%%%%%
\begin{theorem}\label{thm:a3}
%%%%%%%%%%%%%%%%%%
Both the bounded complex-valued measurable potential $q$ and the boundary parameter $\alpha$ are uniquely determined by the spectral pair $(\nu,\psi)$. 
\end{theorem}

Let us give some heuristics to support the thesis that $(\nu,\psi)$ is the natural generalisation of the spectral measure. One should think of $(\nu,\psi)$ not as a way of parameterising $H$ through its \emph{spectrum}, but as a notion closely related to its \emph{polar decomposition}
\[
H=V\abs{H},
\]
where $\abs{H}=\sqrt{H^*H}$ and $V$ is a partial isometry. Theorem~\ref{thm.diagh} below shows that $\nu$ is supported on the spectrum of $\abs{H}$, while $\abs{\psi}$ determines the spectral multiplicity of $\abs{H}$ (which may take values one and two).

The argument of $\psi$ is harder to interpret, but it should be thought of as relating to some infinitesimal angles in the~Hilbert space, generated by the action of the partial isometry $V$; cf.~\cite[Sec.~2.6]{PS1}. In Theorem~\ref{thm.5.4} below we give a precise formula for $\psi(\lambda)$ in one particular case when $\lambda$ is a~simple eigenvalue of $\abs{H}$.

The function $\psi\in L^\infty(\nu)$ is defined up to sets of $\nu$-measure zero. In the next theorem, we fix a representative $\psi$ such that $\psi(s)$ is defined for all $s\in\bbR$, is odd and satisfies $\abs{\psi(s)}\leq1$ for all $s\in\bbR$. Below, $L^{2}_{\Sigma}(\bbR_+;\bbC^{2})$  is the space of $\bbC^2$-valued functions on $\bbR_{+}$ with the inner product
\[
\jap{F,G}=\int_{\bbR_{+}}\jap{\dd\Sigma(\lambda)F(\lambda),G(\lambda)}_{\bbC^2},
\]
where $\jap{\cdot,\cdot}_{\bbC^2}$ is the Euclidean inner product on $\bbC^{2}$.
Let us stress that the integration interval $\bbR_{+}$ does not contain $0$, hence the above integral remains unchanged if $0$ is a point mass of $\Sigma$.
The following statement is proven in Section~\ref{sec.b3}.

\begin{theorem}\label{thm.diagh}
Let $H$ be defined by \eqref{eq.r4a} and \eqref{eq.r4b}. Then the kernel of $H$ is either trivial or one-dimensional. The restriction of $\abs{H}$ onto the orthogonal complement to its kernel is unitarily equivalent to the operator of multiplication by the independent variable in the space $L^{2}_{\Sigma}(\bbR_+;\bbC^{2})$. Moreover, the spectrum of $\abs{H}$ has 
\begin{align}
\text{multiplicity one on } S_1=\{s>0 \mid \abs{\psi(s)}=1\},
\label{eq.multone}
\\
\text{multiplicity two on } S_2=\{s>0 \mid \abs{\psi(s)}<1\}.
\label{eq.multtwo}
\end{align}
\end{theorem}

The following theorem characterises the spectral data of the adjoint operator $H^*$. 

\begin{theorem}\label{thm:a5}
Let $(\nu,\psi)$ be the spectral data of $H$. Then $(\nu,\overline{\psi})$ is the spectral data of $H^*$. Consequently, $H$ is self-adjoint  if and only if $\psi$ is real-valued. 
\end{theorem}
We note that $H$ is self-adjoint if and only if $q$ is real-valued and $\alpha\in\bbR\cup\{\infty\}$.

%%%%%%%%%%%%%%%%%%%%%%%%%%%%%%%%%%%%
\subsection{The spectral pair for self-adjoint and normal $H$}
%%%%%%%%%%%%%%%%%%%%%%%%%%%%%%%%%%%%

If $q$ is real-valued and $\alpha\in\bbR\cup\{\infty\}$, then the operator $H$ is self-adjoint and therefore one can define the (scalar) spectral measure $\sigma$ of $H$, see \eqref{eq.r4}. The next theorem relates $\sigma$ to $(\nu,\psi)$.

\begin{theorem}\label{thm:a6}
If $q$ is real-valued and $\alpha\in\bbR\cup\{\infty\}$, we have 
\begin{equation}
\boxed{
\dd\sigma=(1+\psi)\dd\nu.
}
\label{eq.r6}
\end{equation}
Furthermore, $H$ is self-adjoint and positive semi-definite if and only if $\psi(s)=1$ for $\nu$-a.e. $s>0$.
\end{theorem}
Recall that $\psi$ is odd. Thus, if $H$ is self-adjoint and positive semi-definite, then $\psi(s)=-1$ for $\nu$-a.e. $s<0$ and so $\sigma$ vanishes on the negative half-line. 

Theorems~\ref{thm:a5} and~\ref{thm:a6} are proven in Section~\ref{sec.r}.

A simple calculation shows that the Schr\"odinger operator $H$ with a complex potential $q$ is \emph{normal} (i.e. $H^*H=HH^*$) if and only if $\Im q$ is constant and $\alpha\in\bbR\cup\{\infty\}$. 
Below we describe the spectral pairs for operators of this class. 

If $\sigma$ is a (scalar) measure on $\bbR$, we will denote by $\sigma_*$ the ``reflected'' measure corresponding to the change of variable $x\mapsto -x$, i.e. 
\[
\sigma_*([a,b])=\sigma([-b,-a]).
\]
%%%%%%%%%%%%%%%%%
\begin{theorem}\label{thm:a7}
%%%%%%%%%%%%%%%%%
Let $q$ be a bounded real-valued potential, $\alpha\in\bbR\cup\{\infty\}$, and $\sigma$ be the spectral measure of the corresponding self-adjoint Schr\"odinger operator $H$. For a constant $\omega\in\bbR$, let $(\nu,\psi)$ be the spectral pair of the Schr\"odinger operator $H+\ii\omega I$, corresponding to the potential $q+\ii\omega$ and the boundary parameter $\alpha$. Then $\nu$ is supported on $(-\infty,-\abs{\omega}]\cup[\abs{\omega},\infty)$ and the spectral pair $(\nu,\psi)$ can be determined from the formulas
\begin{align}
\dd\nu(s)&=\frac1{2}\left(\dd\sigma(\lambda)+\dd\sigma_*(\lambda)\right), \label{eq:X3} \\[7pt]
 \psi(s)\dd\nu(s)&=\frac{\lambda+\ii\omega}{2s}\dd\sigma(\lambda)-\frac{\lambda-\ii\omega}{2s}\dd\sigma_*(\lambda), \label{eq:X4}
\end{align}
\vskip6pt 
\noindent where $\lambda\geq0$ and $s\geq\abs{\omega}$ are related by $s^2=\lambda^2+\omega^2$.
\end{theorem}

\begin{remark}
With the aid of~\eqref{eq:X3}, we can express separately the real and imaginary part of~\eqref{eq:X4} as 
\[
\Re\psi(s)\dd\nu(s)=\frac{\lambda}{2s}\left(\dd\sigma(\lambda)-\dd\sigma_*(\lambda)\right)
\quad\mbox{ and }\quad
\Im\psi(s)=\frac{\omega}{s}.
\]
\end{remark}

If $\omega=0$, we obtain $\Im \psi=0$, $\lambda=s$ and 
\begin{align*}
\dd\nu(\lambda)&=\frac12\left(\dd\sigma(\lambda)+\dd\sigma_*(\lambda)\right),
\\
\psi(\lambda)\dd\nu(\lambda)&=\frac12\left(\dd\sigma(\lambda)-\dd\sigma_*(\lambda)\right)
\end{align*}
for $\lambda>0$. This is an equivalent form of the identity \eqref{eq.r6} of Theorem~\ref{thm:a6}. 
Both Theorems~\ref{thm:a6} and \ref{thm:a7}  are proved in Section~\ref{sec.r}.

%%%%%%%%%%%%%%%%%%%%%%%%%%%%%%%%%%%%
\subsection{High energy asymptotics of the spectral data}
%%%%%%%%%%%%%%%%%%%%%%%%%%%%%%%%%%%%
As already mentioned, the description of the image of the spectral map is a delicate question even in the self-adjoint case; this question is outside the scope of this paper. Here we discuss a simpler question. Some constraints on the image of the spectral map in the self-adjoint case are induced by the known asymptotic behavior of the spectral measure at infinity. In \cite{M}, Marchenko proved that in the self-adjoint case for $q\in L_{\mathrm{loc}}^{1}(\bbR_{+})$ limit-point at infinity the spectral measure satisfies, as $r\to\infty$, 
\begin{equation}
 \sigma([0,r])=
 \begin{cases}
 \displaystyle \frac{2}{\pi}(1+\alpha^{2})r^{1/2}+\smallO(r^{1/2}), &\quad \mbox{ if } \alpha\neq\infty, \\[12pt]
 \displaystyle \frac{2}{3\pi}r^{3/2}+\smallO\left(r^{3/2}\right), &\quad \mbox{ if } \alpha=\infty.
 \end{cases}
 \label{eq:marchenko_asympt}
\end{equation}
The left-hand side can be replaced by $\sigma((-\infty,r])$ since $\sigma((-\infty,0))<\infty$.

We will deduce analogous asymptotic relations for spectral pair $(\nu,\psi)$ in the non-self-adjoint case. In fact, we prove  the following more general statement, where we use the notation
\[
 |\Sigma|(r):=\begin{cases}
  \Sigma([0,r]), &\quad\mbox{ if } r\geq0,\\
  \Sigma([r,0]), &\quad\mbox{ if } r<0.
 \end{cases}
\]

\begin{theorem}\label{thm:asympt_Sigma}
%%%%%%%%%%%%%%%%%%
The spectral measure $\Sigma$ of the operator $\bH$ defined by~\eqref{a0} and~\eqref{eq:dom_bH} with a bounded measurable $2\times 2$ Hermitian matrix-valued potential~$Q$ and a boundary parameter $\alpha\in\bbC\cup\{\infty\}$ satisfies
\[
\lim_{r\to\infty}\frac{|\Sigma|(\pm r)}{r^{1/2}}=\frac{1+|\alpha|^{2}}{\pi}
\begin{pmatrix}
1 & \pm 1 \\ \pm 1 & 1
\end{pmatrix}, \quad \mbox{ if } \alpha\neq\infty
\]
and
\[
\lim_{r\to\infty}\frac{|\Sigma|(\pm r)}{r^{3/2}}=\frac{1}{3\pi}
\begin{pmatrix}
1 & \pm 1 \\ \pm 1 & 1
\end{pmatrix},
 \quad \mbox{ if } \alpha=\infty.
\]
\end{theorem}
Observe that the above asymptotics is independent of the argument of $\alpha$. 
Theorem~\ref{thm:asympt_Sigma} is proved in Section~\ref{sec.t}. An immediate corollary of Theorems~\ref{thm:a2} and~\ref{thm:asympt_Sigma} are asymptotic formulas for the spectral pair of $H$.

\begin{theorem}\label{thm:t2}
Let $H$ be the Schr\"odinger operator with a bounded complex-valued potential $q$ and boundary parameter $\alpha\in\bbC\cup\{\infty\}$. The spectral pair $(\nu,\psi)$ of $H$ satisfy
\[
\lim_{r\to\infty}\frac{\nu([0,r])}{r^{1/2}}=\lim_{r\to\infty}\frac{1}{r^{1/2}}\int_{0}^{r}\psi(s)\dd\nu(s)=\frac{1+|\alpha|^{2}}{\pi}, \quad \mbox{ if } \alpha\neq\infty
\]
and
\[
\lim_{r\to\infty}\frac{\nu([0,r])}{r^{3/2}}=\lim_{r\to\infty}\frac{1}{r^{3/2}}\int_{0}^{r}\psi(s)\dd\nu(s)=\frac{1}{3\pi}, \quad \mbox{ if } \alpha=\infty.
\]
\end{theorem}

It follows from the limit formulas of Theorem~\ref{thm:t2} that $\psi(s)\to1$ ``in an average sense" as $s\to\infty$. Recalling \eqref{eq.r6}, we see that Theorem~\ref{thm:t2} agrees with the asymptotics \eqref{eq:marchenko_asympt} in the self-adjoint case.

\subsection{The spectral pair at a simple eigenvalue of $\abs{H}$}

Our last result provides expressions for $\nu(\{\lambda\})$ and $\psi(\lambda)$ in the case when $\lambda>0$ is a simple eigenvalue of $\abs{H}$. These formulas are given in terms of a distinguished solution to the corresponding differential ``anti-linear eigenvalue equation" defined in the following lemma.

\begin{lemma}\label{lma.b1}
Let $\lambda>0$ be a simple eigenvalue of $|H|$. Then there exists a~unique, up to multiplication by $\pm1$, normalised function $e\in L^2(\bbR_+)$, which satisfies the equation
\begin{equation}
-e''+qe=\lambda\overline{e}
\label{f1}
\end{equation}
(observe the complex conjugation in the right-hand side!)
and the boundary condition \eqref{eq.r4b}, i.e.
\[
e'(0)+\alpha e(0)=0.
\]
We will call such $e$ the \underline{distinguished solution}. 
\end{lemma}

For a function $f$ from the domain of $H$, we define 
\begin{equation}
\ell_\alpha(f):=\frac{\overline{\alpha}f'(0)-f(0)}{\sqrt{1+|\alpha|^{2}}}
\label{e11a}
\end{equation}
if $\alpha\in\bbC$, and $\ell_{\infty}(f):=f'(0)$ if $\alpha=\infty$.

\begin{theorem}\label{thm.5.4}
Let $\lambda>0$ be a simple eigenvalue of $\abs{H}$ and let $e$ be the distinguished solution from Lemma~\ref{lma.b1}. Then $\ell_\alpha(e)\not=0$ and 
\begin{equation}
\boxed{
\nu(\{\lambda\})=
\frac{1}{2}
\abs{\ell_\alpha(e)}^2,
\quad
\psi(\lambda)=
\frac{\overline{\ell_\alpha(e)^2}}{\abs{\ell_\alpha(e)}^2}.
}
\label{e11}
\end{equation}
In particular, $\abs{\psi(\lambda)}=1$. 
\end{theorem}
Proofs of Lemma~\ref{lma.b1} and Theorem~\ref{thm.5.4} are given in Section~\ref{sec.f}.

%%%%%%%%%%%%%%%%%%%%%%%%%%%%%%%%%%%%
\subsection{Organisation of the paper}
%%%%%%%%%%%%%%%%%%%%%%%%%%%%%%%%%%%%
In Section~\ref{sec.disc} we discuss some related results. 
In Section~\ref{sec:a3}, we briefly recall key facts from the Titchmarsh--Weyl theory and the spectral theory of $\bH$.  In Sections~\ref{sec.d} and~\ref{sec.c}, the Borg--Marchenko type Theorem~\ref{thm:a1} is proven. Proofs of Theorems~\ref{thm:a2} and~\ref{thm.diagh} are given in Section~\ref{sec.b}. In Section~\ref{sec.r}, the case of self-adjoint and normal $H$ is considered and Theorems~\ref{thm:a5}, \ref{thm:a6} and \ref{thm:a7} are proved. In Section~\ref{sec:examp}, as a concrete example, we compute the spectral pair for the free Schr{\"o}dinger operator. Asymptotic formulas of Theorem~\ref{thm:asympt_Sigma} are deduced in Section~\ref{sec.t}. Formulas for the spectral pair at a simple eigenvalue of $|H|$ of Theorem~\ref{thm.5.4} as well as  Lemma~\ref{lma.b1} are proven in Section~\ref{sec.f}. Concluding remarks concerning possible extensions of our results are formulated in Section~\ref{sec:misc}.  Finally, the paper is concluded by an appendix, where key results from the spectral theory of $\bH$ are proved for completeness. Although these results are known to experts, extracting precise proofs from the existing literature is not always a~simple task.

%%%%%%%%%%%%%%%%%%%%%%%%%%%%%%%%%%%%
\subsection{Acknowledgements}
%%%%%%%%%%%%%%%%%%%%%%%%%%%%%%%%%%%%
We are grateful to Fritz Gesztesy for very useful discussions.

%%%%%%%%%%%%%%%%%%%%%%%%%%%%%%%%%%%%
%%%%%%%%%%%%%%%%%%%%%%%%%%%%%%%%%%%%
\section{Related literature}\label{sec.disc}
%%%%%%%%%%%%%%%%%%%%%%%%%%%%%%%%%%%%
%%%%%%%%%%%%%%%%%%%%%%%%%%%%%%%%%%%%

%%%%%%%%%%%%%%%%%%%%%%%%%%%%%%%%%%%%
\subsection{Comparison with Jacobi matrices}
%%%%%%%%%%%%%%%%%%%%%%%%%%%%%%%%%%%%

We are strongly motivated by our previous work \cite{PS1}, where we have introduced the notion of spectral pair and studied the corresponding spectral map for a class of non-self-adjoint Jacobi matrices. Here we recall the key points of \cite{PS1} (see also subsequent works~\cite{PS2} and \cite{ELY}) for the purposes of comparison with the current paper. 

Jacobi matrices can be thought of as a discrete variant of second order differential operators on the half-line. Below for brevity, we will refer to the theory of Jacobi matrices as \emph{discrete case} and to the theory of Schr{\" o}dinger operators as \emph{continuous case}. 

For two bounded sequences $\{a_{n}\}_{n=0}^{\infty}$ and $\{b_{n}\}_{n=0}^{\infty}$, we denote
\[
J:=\begin{pmatrix}
b_0 & a_0 & 0 & 0 & 0 &\cdots\\
a_0 & b_1 & a_1 & 0 & 0 & \cdots\\
0 & a_1 & b_2 & a_2 & 0 &\cdots\\
0 & 0 & a_2 & b_3 & a_3 & \cdots\\
\vdots&\vdots&\vdots&\vdots&\vdots&\ddots
\end{pmatrix}.
\]
This defines  $J$ as a bounded operator on $\ell^{2}(\bbN_{0})$, $\bbN_{0}=\{0,1,2,\dots\}$; we will comment on the boundedness assumption below. Let us start with the self-adjoint case:
\begin{equation}
a_n>0\quad\text{ and }\quad b_n\in\bbR.
\label{e11b}
\end{equation}
It is well-known and easy to check that the first vector $\delta_{0}=(1,0,0,\dots)$ of the standard basis in $\ell^{2}(\bbN_{0})$ is a cyclic element of $J$. 
The corresponding spectral measure $\sigma$ is uniquely defined by the equation
\[
\jap{f(J)\delta_0,\delta_0}_{\ell^{2}(\bbN_{0})}
=
\int_{\bbR}f(x)\dd\sigma(x)
\]
for all $f\in C(\bbR)$. Thus defined, $\sigma$ is a scalar probability measure on $\bbR$ with the support that is compact (because $J$ is bounded) and infinite (as a set). It is a standard result of the theory of Jacobi matrices (see e.g. \cite{Akh}) that the \emph{spectral map} 
\[
\bigl(\{a_{n}\}_{n=0}^{\infty},\{b_{n}\}_{n=0}^{\infty}\bigr)
\mapsto \sigma
\]
is a one-to-one correspondence between the set of bounded sequences satisfying \eqref{e11b} and the set of probability measures on $\bbR$ with compact infinite support. 

Now we turn to the non-self-adjoint case considered in \cite{PS1}. We relax the self-adjointness condition \eqref{e11b} to 
\begin{equation}
a_n>0\quad\text{ and }\quad b_n\in\bbC.
\label{e11c}
\end{equation}
The hermitisation of $J$, i.e. the operator 
\[
 \hskip12pt 
 \begin{pmatrix}
  0 & J \\ J^{*} & 0
  \end{pmatrix}
\quad \text{ on }\; \ell^{2}(\bbN_{0})\oplus\ell^{2}(\bbN_{0}),
\]
is unitarily equivalent to the self-adjoint block Jacobi matrix
\[
\bJ=\begin{pmatrix}
B_0 & A_0 & 0 & 0 & 0 &\cdots\\
A_0 & B_1 & A_1 & 0 & 0 & \cdots\\
0 & A_1 & B_2 & A_2 & 0 &\cdots\\
0 & 0 & A_2 & B_3 & A_3 & \cdots\\
\vdots&\vdots&\vdots&\vdots&\vdots&\ddots
\end{pmatrix} \quad\mbox{ with }\quad
\substack{\displaystyle
  A_{n}=\begin{pmatrix}
   0 & a_{n} \\ a_{n} & 0
  \end{pmatrix},
  \\[12pt]
  \displaystyle
  B_{n}=\begin{pmatrix}
   0 & b_{n} \\ \overline{b_{n}} & 0
  \end{pmatrix},
 }
\]
acting on $\ell^{2}(\bbN_{0};\bbC^{2})$. Denoting by $\mathbf{P}_0:\ell^{2}(\bbN_{0};\bbC^{2})\to\bbC^{2}$ the projection onto the first component, the $2\times 2$ matrix-valued \emph{spectral measure} of $\bJ$ can be defined by the formula
\[
 \Sigma:=\mathbf{P}_0 E_{\bJ}\mathbf{P}_0^{*},
\]
where $E_{\bJ}$ is the projection-valued spectral measure of $\bJ$. Since $\bJ$ enjoys symmetries analogous to $\bH$, one can prove exactly the same formula as~\eqref{eq:a1}
with an even positive measure $\nu$ on $\bbR$ and an odd complex-valued function $\psi\in L^\infty(\nu)$ satisfying $\abs{\psi(s)}\leq1$ for $\nu$-a.e. $s\in\bbR$. 
The support of $\nu$ is compact and infinite. 

The main result of \cite{PS1} says that the spectral map 
\begin{equation}
\bigl(\{a_{n}\}_{n=0}^{\infty},\{b_{n}\}_{n=0}^{\infty}\bigr)
\mapsto (\nu,\psi)
\label{e11d}
\end{equation}
is a bijection between the set of all bounded sequences satisfying \eqref{e11c} and the set of all pairs $(\nu,\psi)$, where $\nu$ is an even measure with compact infinite support and $\psi\in L^\infty(\nu)$ is an odd complex-valued function satisfying $\abs{\psi(s)}\leq1$ for $\nu$-a.e. $s\in\bbR$. 

The injectivity part of this statement is the discrete analogue of Theorem~\ref{thm:a3}. The surjectivity statement is only known in the discrete case. 

Another (shorter) proof of the bijectivity of the spectral map \eqref{e11d} was given in~\cite[Sec.~3]{PS2}. The discrete analogue of Theorem~\ref{thm.diagh} is not explicitly stated in \cite{PS1}, but it was established in the follow-up paper \cite{PS2} in an abstract setting. Discrete analogues of Theorem~\ref{thm:a5} and~\ref{thm:a6} can be found in \cite[Thms.~2.6 and~2.7]{PS1}. 

We mention one difference between the discrete and continuous cases. It was shown in~\cite[Thm.~2.6(iii)]{PS1} that $J$ has vanishing diagonal $b_n=0$ if and only if $\psi=0$. In contrast to this, in the continuous case $\psi$ never vanishes identically, as it would contradict Theorem~\ref{thm:t2}.

The condition of the boundedness of $\{a_{n}\}_{n=0}^{\infty}$ and $\{b_{n}\}_{n=0}^{\infty}$ can be lifted. The spectral map \eqref{e11d} for unbounded non-self-adjoint Jacobi matrices $J$ was considered in the recent paper \cite{ELY}. 

The spectral pair $(\nu,\psi)$ can be related to $J$ more directly via formulas
\begin{align}
\langle f(|J|)\delta_{0},\delta_{0} \rangle_{\ell^{2}(\bbN_{0})}&=\int_{\bbR_{+}}f(s)\dd\nu(s), 
\label{e11e}
\\
\langle Jf(|J|)\delta_{0},\delta_{0} \rangle_{\ell^{2}(\bbN_{0})}&=\int_{\bbR_{+}}sf(s)\psi(s)\dd\nu(s)
\label{e11f}
\end{align}
valid for all $f\in C([0,\infty))$. This allows one to define the spectral pair in the discrete case more directly, bypassing the hermitisation of $J$; this approach is used in \cite{PS1}. 
It is possible to use the same approach to the definition of spectral data in the continuous case, if instead of $\delta_0$ one uses a suitable linear combination of the delta-function at the origin and its derivative. However, we prefer to use the more traditional language of spectral theory involving the Titchmarsh--Weyl $m$-function. 

%%%%%%%%%%%%%%%%%%%%%%%%%%%%%%%%%%%%
\subsection{Complex-symmetric structure}
%%%%%%%%%%%%%%%%%%%%%%%%%%%%%%%%%%%%
At the heart of our definition of spectral pair is the complex-symmetric structure. 
Let $\calC$ be the operator of complex conjugation acting on $L^2(\bbR_+)$ as $\calC f=\overline{f}$. A linear operator $A$ on $L^2(\bbR_+)$ is called complex-symmetric if it satisfies
\[
\calC A=A^*\calC.
\]
The same definition can be applied to operators on $\ell^2(\bbN_0)$. It is clear that the Schr{\"o}dinger operator $H$ with a complex potential is complex-symmetric in $L^2(\bbR_+)$, and the Jacobi matrices $J$ satisfying \eqref{e11c} are complex-symmetric in $\ell^2(\bbN_0)$.

Let $A$ be a bounded complex-symmetric operator and let $\delta_0$ be a vector satisfying $\calC \delta_0=\delta_0$. Then one can define the spectral data $(\nu,\psi)$ similarly to \eqref{e11e}, \eqref{e11f}. This viewpoint is further explored in \cite{PS2}. A general theory of complex-symmetric operators is surveyed in~\cite{gar-put_06,gar-put_07}.

Another class of operators satisfying the complex-symmetric condition is the class of Hankel operators. Our idea of using the pair $(\nu,\psi)$ for the spectral data is inspired by the inverse spectral theory for compact Hankel operators developed by G\'erard and Grellier; see \cite{GG1,GG2} and references therein. The spectral data used by G\'erard and Grellier resembles the pair $(\nu,\psi)$ (in fact, the $\psi$ component is much more complicated in their case); see also the follow-up paper \cite{GPT}.

%%%%%%%%%%%%%%%%%%%%%%%%%%%%%%%%%%%%
\subsection{Marchenko's theory}
%%%%%%%%%%%%%%%%%%%%%%%%%%%%%%%%%%%%
It would be remiss not to mention the fundamental work of Marchenko \cite{M2}, see also the earlier paper~\cite{MR} and the book~\cite{M-book}, where he developed a~different kind of inverse spectral theory for non-self-adjoint Schr\"odinger operators $H$. Marchenko associates with $H$ a \emph{spectral function}, which is a distribution on the real line related to the solution $\varphi(x,\lambda)$ for real $\lambda$. He proves the Borg--Marchenko type uniqueness theorem and describes the set of admissible spectral functions. Although Marchenko's theory is conceptually close to our work, we do not pursue a~connection between Marchenko's spectral function and our spectral pair.

%%%%%%%%%%%%%%%%%%%%%%%%%%%%%%%%%%%%
\subsection{Further work}
%%%%%%%%%%%%%%%%%%%%%%%%%%%%%%%%%%%%
Much work on inverse spectral problems for differential operators, in both self-adjoint and non-self-adjoint case, has been done by Yurko and his collaborators, see~\cite{Y1, Y2,FY1, FY2}. Yurko studies the concept of the Weyl matrix and proves uniqueness theorems using the \emph{method of spectral mappings}. 

In \cite{BPW}, the authors prove a local version of uniqueness theorem for complex potentials.

%%%%%%%%%%%%%%%%%%%%%%%%%%%%%%%%%%%%
%%%%%%%%%%%%%%%%%%%%%%%%%%%%%%%%%%%%
\section{Preliminaries: the Titchmarsh--Weyl theory for $\bH$}
\label{sec:a3}
%%%%%%%%%%%%%%%%%%%%%%%%%%%%%%%%%%%%
%%%%%%%%%%%%%%%%%%%%%%%%%%%%%%%%%%%%
Here we briefly recall key facts from the Titchmarsh--Weyl theory for the operators $\bH$ of the form \eqref{a0} with a~general $2\times 2$ Hermitian matrix-valued potential $Q$ and a boundary condition \eqref{eq:a2a}, \eqref{eq:a2b}. This theory has been developed during the second half of the 20th century in much greater generality (for so-called Hamiltonian systems) and can now be considered as standard.

\subsection{$\bH$ is in the limit point case}
For $\lambda\in\bbC$, we consider the eigenvalue equation
\begin{equation}
-\epsilon F''(x,\lambda)+Q(x)F(x,\lambda)=\lambda F(x,\lambda), \quad x>0.
\label{eq:init_diff_eq}
\end{equation}
Here, depending on the context, $F(x,\lambda)$ is either a $2$-component column vector or a $2\times2$ matrix. The boundedness of $Q$ implies that this equation is in the \emph{limit-point case} at $+\infty$. In other words, we have the following fundamental fact, see~\cite[Lem.~2.1 and Cor.~2.1]{HS_81}. 

\begin{lemma}[The dimension lemma]\label{lem:dim_S}
Let $\lambda\in\bbC\setminus\bbR$. The dimension of the space
\begin{equation}
S(\lambda):=\{F\in L^{2}(\bbR_{+};\bbC^{2}) \mid F \mbox{ is a solution of~\eqref{eq:init_diff_eq}}\}
\label{eq.dim}
\end{equation}
equals $2$.
\end{lemma}
\begin{proof}
The dimension of $S(\lambda)$ for $\lambda$ in the upper and lower half-planes coincides with the deficiency indices of the operator $\bH$ restricted to the space $C_{0}^{\infty}(\bbR_+;\bbC^2)$. Since the deficiency indices are invariant under the bounded perturbation by $Q$, it suffices to find them for $Q=0$. In this case, one checks that $\dim S(\lambda)=2$ by a~direct computation.
\end{proof}

\subsection{The Titchmarsh--Weyl $M$-function}\label{sec.b2}
Let us denote by $\Phi,\Theta$ the $2\times2$ matrix-valued solutions (the fundamental system) of the eigenvalue equation \eqref{eq:init_diff_eq} satisfying the boundary conditions
\begin{equation}
\begin{aligned}
\Phi(0,\lambda)&=S,\\
\epsilon\Phi'(0,\lambda)&=-C,
\end{aligned}
\qquad
\begin{aligned}
\Theta(0,\lambda)=C, \\
\epsilon\Theta'(0,\lambda)=S,
\end{aligned}
\quad
\label{eq:phi,theta_bc_cond}
\end{equation}
where $\lambda\in\bbC$ and the matrices $S$, $C$ are as in \eqref{eq:a2b}. 
We note that $\Phi$ (but not $\Theta$) satisfies the boundary condition \eqref{eq:a2a}. The following statement is the direct analogue of the definition of the scalar $m$-function.  The proof and relevant literature is discussed in Appendix~\ref{sec:app1}.

%%%%%%%%%%%%%%%%%%%%%
\begin{proposition}\label{prp.Mf}
%%%%%%%%%%%%%%%%%%%%%
For every $\lambda\in\bbC\setminus\bbR$, there exists a unique $2\times 2$ matrix $M_{\alpha}(\lambda)$ such that both columns of the matrix
\begin{equation}
X(x,\lambda)=\Theta(x,\lambda)-\Phi(x,\lambda)M_\alpha(\lambda)
\label{eq:X}
\end{equation}
belong to $L^2(\bbR_+;\bbC^2)$ as a function of $x$. The matrix-valued function $M_\alpha$ is analytic in $\bbC\setminus\bbR$, satisfies the identity 
\begin{equation}
M_\alpha(\overline{\lambda})=M_\alpha(\lambda)^{*},
\quad 
\lambda\in\bbC\setminus\bbR,
\label{eq:a4}
\end{equation}
and
\begin{equation}
\Im M_\alpha(\lambda)>0, \quad \text{ if }\Im \lambda>0.
\label{eq:a4a}
\end{equation}
\end{proposition}

In \eqref{eq:a4a}, $M>0$ means that the matrix $M$ is positive definite and we use the notation 
\[
\Re M =\frac12(M+M^{*}) \quad\mbox{ and }\quad \Im M = \frac1{2\ii}(M-M^{*}). 
\] 
Inspecting the initial conditions for the solution $X(x,\lambda)$ of \eqref{eq:X}, we find that its columns are linearly independent and so they span the 2-dimensional space $S(\lambda)$ of \eqref{eq.dim}. Thus, the uniqueness of $M_\alpha(\lambda)$ follows immediately from the dimension lemma. 

Proposition~\ref{prp.Mf} implies that $M_{\alpha}$ is a matrix-valued Herglotz--Nevanlinna function and therefore, by the general integral representation theorem, see~\cite[Theorems~2.3 and 5.4]{GT}, we have 
\begin{equation}
M_\alpha(\lambda)=
\Re M_\alpha(\ii)+\int_{-\infty}^\infty \left(\frac{1}{t-\lambda}-\frac{t}{1+t^2}\right)\dd\Sigma(t), \quad 
\lambda\in\bbC\setminus\bbR,
\label{eq:intres}
\end{equation}
where $\Sigma$ is a unique $2\times 2$ matrix-valued measure on $\bbR$ such that the integral
\[
 \int_{-\infty}^{\infty}\frac{\dd\Sigma(t)}{1+t^{2}}
\]
converges. The measure $\Sigma$ is called the \emph{spectral measure} of $\bH$.

\subsection{Asymptotic formula for the $M$-function}
We denote by $P_+$ and $P_-$ the orthogonal projections 
\begin{equation}
P_{+}:=\frac{1}{2}\begin{pmatrix}1&1\\1&1\end{pmatrix}
\quad\text{ and }\quad
P_{-}:=\frac{1}{2}\begin{pmatrix}1&-1\\-1&1\end{pmatrix}\, .
\label{eq:Ppm}
\end{equation}

\begin{proposition}\label{prop:M_func_asympt}
Let $\alpha\in\bbC\cup\{\infty\}$ and $\lambda=k^2$ with $\Re k>0$, $\Im k>0$. Then as $\abs{k}\to\infty$ along any ray, the $M$-function satisfies
\begin{equation}
M_\alpha(\lambda)=
\epsilon A+\dfrac{1+|\alpha|^{2}}{k}(\ii P_{+}-P_{-})+\dfrac{1+\abs{\alpha}^2}{k^2}(\ii P_{+}-P_{-})\epsilon A(\ii P_{+}-P_{-})+\mathcal{O}(1/k^3),
\label{eq.Masymp1}
\end{equation}
if $\alpha\not=\infty$, and 
\begin{equation}
M_\infty(\lambda)=k(\ii P_{+}+P_{-})+\mathcal{O}(1/k), 
\label{eq.Masymp2}
\end{equation}
if $\alpha=\infty$. 
\end{proposition}
If $\alpha=0$, the expression \eqref{eq.Masymp1}  simplifies to 
\begin{equation}
M_0(\lambda)=\dfrac{1}{k}(\ii P_{+}-P_{-})+\mathcal{O}(1/k^3).
\label{eq.Masymp3}
\end{equation}
Explicitly, the matrices appearing in the above asympotics are 
\begin{equation}
\epsilon A
=\begin{pmatrix}0&\alpha \\ \overline{\alpha}&0\end{pmatrix},
\quad
(\ii P_{+}-P_{-})\epsilon A(\ii P_{+}-P_{-})
=
-
\begin{pmatrix}\Re\alpha&-\Im\alpha \\ \Im\alpha&\Re\alpha\end{pmatrix}.
\label{eq.Masymp4}
\end{equation}
The proof of Proposition~\ref{prop:M_func_asympt} and comments on the literature are given in the Appendix~\ref{sec.masympt}. We use the boundedness of $Q$ in the proof. However, similar asymptotics, with somewhat worse error estimates, is known for more general potentials; see for example \cite[Thm.~3.2]{And}. 

It is important for us that $\alpha$ can be determined from the asymptotic coefficients in the right hand side of \eqref{eq.Masymp1} and that the error term in \eqref{eq.Masymp2} tends to $0$ as $|k|\to\infty$. 

\subsection{The spectral theorem}
The measure $\Sigma$ is the spectral measure of $\bH$ in the following sense. 
Let $L^{2}_{\Sigma}(\bbR;\bbC^{2})$ be the space of all $\bbC^2$-valued functions on $\bbR$ equipped with the inner product
\[
\jap{F,G}=\int_{-\infty}^\infty \jap{\dd\Sigma(\lambda)F(\lambda),G(\lambda)}_{\bbC^2}.
\]
Define the operator $U: L^2(\bbR_+;\bbC^2)\to L_{\Sigma}^{2}(\bbR;\bbC^{2})$ by
\[
(UF)(\lambda):=\int_0^\infty \Phi(x,\lambda)^{*}F(x)\dd x
\]
initially on the set of smooth compactly supported functions $F$. Also, let $\Lambda$ be the operator of multiplication by the independent variable in $L_{\Sigma}^{2}(\bbR;\bbC^{2})$. 
%%%%%%%%%%%%%%%%%%%%%
\begin{proposition}\label{thm.sp.decomp}
%%%%%%%%%%%%%%%%%%%%%
The operator $U$ extends to a unitary map from $L^2(\bbR_+;\bbC^2)$ onto $L_{\Sigma}^{2}(\bbR;\bbC^{2})$. 
Moreover, $U$ intertwines $\bH$ with $\Lambda$, i.e. $U\Dom \bH\subset \Dom\Lambda$ and 
\begin{equation}
U\bH F=\Lambda UF
\label{eq:intertwine}
\end{equation}
for all $F\in\Dom \bH$.
\end{proposition}
See Appendix~\ref{sec:app3} for the proof.

\subsection{Non-singularity of $\Phi$, $\Theta$, $X$ for non-real $\lambda$}
We will frequently need to use the inverses of matrices $\Phi$, $\Theta$, $X$ and their derivatives with respect to $x$, and therefore we need to know that these matrices are non-singular. The following lemma is known; it can be deduced, for example, from the more general statement for Hamiltonian systems given in \cite[Lem.~2.2]{HS_81} and the remark after the lemma. For completeness we provide a short proof below.

\begin{lemma}\label{lma.nonsingular}
For any $\lambda\in\bbC\setminus\bbR$ and any $x>0$, the matrices 
\[
\Phi(x,\lambda), \quad \Theta(x,\lambda), \quad X(x,\lambda)
\quad\text{ and }\quad 
\Phi'(x,\lambda),\quad \Theta'(x,\lambda),\quad X'(x,\lambda)
\]
(the derivatives are with respect to $x$) are non-singular. In case of $X(x,\lambda)$ and $X'(x,\lambda)$, this also applies to $x=0$. 
\end{lemma}

\begin{proof}
Let $a>0$. Suppose, for example, that $u\in\bbC^2$ is such that $\Phi(a,\lambda)u=0$. Then $u(x):=\Phi(x,\lambda)u$ solves the equation 
\[
-\epsilon u''(x)+Q(x)u(x)=\lambda u(x), \quad x\in (0,a)
\]
and satisfies the boundary conditions \eqref{eq:a2a} at $x=0$ and $u(a)=0$ at $x=a$. It follows that 
$u(x)=0$ for $x\in(0,a)$ since otherwise $\lambda$ would be a non-real eigenvalue of the corresponding self-adjoint boundary value problem on the interval $(0,a)$.
From here it is easy to see that $u=0$. This proves that $\Phi(a,\lambda)$ is non-singular. The argument for $\Phi'$, $\Theta$, $\Theta'$ goes along the same lines.

Let $a\geq0$ and suppose $\det X(a,\lambda)=0$. In a similar way, this implies that $\lambda$ is a non-real eigenvalue of the operator on the interval $(a,\infty)$ with the Dirichlet boundary condition at $x=a$, which leads to a contradiction. The same argument works for $X'(a,\lambda)$, with the Neumann boundary condition at $x=a$ instead of Dirichlet. 
\end{proof}

\subsection{Constancy of Wronskian}
We will make use of the matrix Wronskian 
\[
 [U(x),V(x)]:=U(x)\epsilon V'(x)-U'(x)\epsilon V(x)\, .
\]
It is standard to check that if $U(x,\lambda)$ and $V(x,\lambda)$ are matrix-valued solutions of the eigenvalue equation~\eqref{eq:init_diff_eq}, then the Wronskian $[U(\cdot,\overline{\lambda})^{*},V(\cdot,\lambda)]$ is independent of $x$. In addition, taking into account the boundary conditions~\eqref{eq:phi,theta_bc_cond} and identities~\eqref{eq:S-C_id}, one readily verifies the formulas
\begin{align}
[\Phi(\cdot,\overline{\lambda})^*,\Phi(\cdot,\lambda)]
&=
[\Theta(\cdot,\overline{\lambda})^*,\Theta(\cdot,\lambda)]
=0,
\label{eq:wronsk_1}
\\
[\Phi(\cdot,\overline{\lambda})^*,\Theta(\cdot,\lambda)]
&=
-[\Theta(\cdot,\overline{\lambda})^*,\Phi(\cdot,\lambda)]
=I,
\label{eq:wronsk_2}
\end{align}
which will be used several times. 

%%%%%%%%%%%%%%%%%%%%%%%%%
\subsection{The case of a compact interval}
%%%%%%%%%%%%%%%%%%%%%%%%%

Let us briefly discuss the spectral problem for the eigenvalue equation \eqref{eq:init_diff_eq} on the compact interval $[0,b]$ with $b<\infty$. Along with the boundary condition \eqref{eq:a2a0} at $x=0$, we set the boundary condition 
\begin{equation}
F'(b)+BF(b)=0,
\label{eq:bc_b}
\end{equation}
with 
\[
 B=\begin{pmatrix}
  \overline{\beta} & 0 \\ 0 & \beta
 \end{pmatrix},
\]
where $\beta\in\mathbb{C}\cup\{\infty\}$. With solutions $\Phi$ and $\Theta$ as before,  the $M$-function is defined by
\begin{equation}
M_\alpha(b;\lambda):=\left(\Phi'(b,\lambda)+B\Phi(b,\lambda)\right)^{-1}\left(\Theta'(b,\lambda)+B\Theta(b,\lambda)\right)
\label{eq:def_M_b}
\end{equation}
for $\lambda\in\bbC\setminus\bbR$ (the inverse exists by the argument of Lemma~\ref{lma.nonsingular}). We suppress the dependence on $\beta$ in our notation. 
It is easy to see that 
\begin{equation}
 X_{b}(x,\lambda):=\Theta(x,\lambda)-\Phi(x,\lambda)M_{\alpha}(b;\lambda)
\label{eq:def_chi_b}
\end{equation}
solves the eigenvalue equation \eqref{eq:init_diff_eq} and satisfies the boundary condition \eqref{eq:bc_b}. The $M$-function $M_\alpha(b;\cdot)$ is a meromorphic Herglotz--Nevanlinna function with poles at the eigenvalues of $\bH$. 

\begin{remark}\label{rem:M_asympt_compact}
The asymptotics of Proposition~\ref{prop:M_func_asympt} remains true for $M_\alpha(b;\lambda)$, independently of the choice of the boundary condition at $b$.
\end{remark}

%%%%%%%%%%%%%%%%%%%%%%%%%%%%%%
%%%%%%%%%%%%%%%%%%%%%%%%%%%%%%
\section{Uniqueness I: solutions to the Cauchy problem}
\label{sec.d}
%%%%%%%%%%%%%%%%%%%%%%%%%%%%%%
%%%%%%%%%%%%%%%%%%%%%%%%%%%%%

\subsection{Overview}
In this section and in the next one, we show that the measure $\Sigma$ uniquely determines the matrix-valued potential $Q$;  this will prove Theorem~\ref{thm:a1}. Since $\Sigma$ uniquely determines the $M$-function $M_\alpha$, the proof reduces to showing that $M_\alpha$ determines $Q$ and the parameter $\alpha$. This is a generalisation of the Borg--Marchenko uniqueness theorem to the vector case. Such generalisations are known, see e.g.~\cite{And} or~\cite{GS}, but only in the case when $\epsilon$ in~\eqref{a0} is the identity matrix. 

We will retrace the classical proof of Borg--Marchenko's theorem, closely following the presentation of Bennewitz \cite{B}. As a first step, we need to establish the asymptotics of solutions $\Phi(x,\lambda)$ and $\Theta(x,\lambda)$ as $\abs{\lambda}\to\infty$ in the complex plane.  Let us explain the subtlety of this problem, which is specific to the case of sign indefinite $\epsilon$. 

Below, $P_+$ and $P_-$ are the matrices introduced in \eqref{eq:Ppm}. 
It is important to observe that these matrices are the projections onto the eigenspaces of $\epsilon$:
\[
\epsilon P_{+}=P_{+}\epsilon=P_{+}, \quad 
\epsilon P_{-}=P_{-}\epsilon=-P_{-}.
\]
In order to explain the heuristics, let us consider the case $Q\equiv0$. Then any solution $F$ to the eigenvalue equation
\[
-\epsilon F''=\lambda F, \quad \lambda=k^2,
\]
with fixed initial conditions $F(0)$, $F'(0)$ can be written as 
\begin{align}
F(x)=P_+F(0)\cos kx+P_-F(0)\cosh kx
+P_+F'(0)\frac{\sin kx}{k}+P_-F'(0)\frac{\sinh kx}{k}.
\label{eq:dd1}
\end{align}
Let $\abs{\lambda}\to\infty$ in the upper half-plane and let $k$ be in the first quadrant.  We observe that in the expression \eqref{eq:dd1} we have two competing leading order exponential terms:
\[
\ee^{kx} \quad \text{and} \quad \ee^{-\ii kx}.
\]
If $\Re k>\Im k$, the first one dominates, and if $\Re k<\Im k$, the second one dominates. Furthermore, in both cases the leading term of the asymptotics of $F(x)$ as $\abs{\lambda}\to\infty$ has a matrix coefficient which is non-invertible because of the presence of the projections $P_{\pm}$. However, for the standard proof of uniqueness that we follow, it is important to have invertibility of the leading term asymptotics of $\Phi(x,\lambda)$ and $\Theta(x,\lambda)$. This forces us to consider the borderline case $\Re k=\Im k$, i.e. $\lambda\to\infty$ in the upper half-plane along the imaginary axis. This motivates the main result of this section. Below $S$, $C$ are the matrices defined in \eqref{eq:a2b}.

\begin{theorem}\label{thm:d1}
Let $\Phi$, $\Theta$ be the solutions to the eigenvalue equation \eqref{eq:init_diff_eq} with the initial conditions~\eqref{eq:phi,theta_bc_cond} and let $\lambda=2\ii\kappa^2$ with $\kappa>0$. Then, as $\kappa\to\infty$, we have
\begin{align}
\ee^{-\kappa x}\,\Phi(x,\lambda)=&
\frac{1}{2}\left(\ee^{-\ii\kappa x}P_{+}+\ee^{\ii\kappa x}P_{-}\right)S\left(1+\calO_x\left(1/\kappa\right)\right)
\notag
\\
&-\frac{1}{2(\kappa+\ii\kappa)}\left(\ii\ee^{-\ii\kappa x}P_{+}-\ee^{\ii\kappa x}P_{-}\right)C\left(1+\calO_x\left(1/{\kappa}\right)\right),
\label{eq:dd2}
\\
\ee^{-\kappa x}\,\Theta(x,\lambda)=&
\frac{1}{2}\left(\ee^{-\ii\kappa x}P_{+}+\ee^{\ii\kappa x}P_{-}\right)C\left(1+\calO_x\left(1/{\kappa}\right)\right)
\notag
\\
&+\frac{1}{2(\kappa+\ii\kappa)}\left(\ii\ee^{-\ii\kappa x}P_{+}-\ee^{\ii\kappa x}P_{-}\right)S\left(1+\calO_x\left(1/{\kappa}\right)\right).
\label{eq:dd3}
\end{align}
\end{theorem}

Of course, the second term in the right-hand side of \eqref{eq:dd2} is significant only if the first term vanishes, i.e. only if $\alpha=\infty$. Similarly, the second term in \eqref{eq:dd3} is significant only if $\alpha=0$. 
We emphasize that the matrices
\[
\ee^{-\ii\kappa x}P_++ \ee^{\ii\kappa x}P_-
\quad\mbox{ and }\quad
\ii\ee^{-\ii\kappa x}P_+- \ee^{\ii\kappa x}P_-
\]
are invertible; in fact, they are unitary. Of course, $S$ and $C$ are also invertible unless $\alpha=0$ or $\alpha=\infty$. 

In the next subsection we prepare some estimates of the Volterra equation that is equivalent to our differential equation \eqref{eq:init_diff_eq}. After that, the proof of Theorem~\ref{thm:d1} is given in Subsection~\ref{subsec:proof-asympt-thm}.

\subsection{Analysis of the Volterra equation}
Fix $k\in\bbC$, $k\not=0$; let $F$ be a $2\times 2$ matrix-valued solution to the equation
\begin{equation}
-\epsilon F''(x)+Q(x)F(x)=k^2F(x), \quad x>0,
\label{eq:d1}
\end{equation}
for some fixed initial conditions $F(0)$ and $F'(0)$. As it is standard, we convert \eqref{eq:d1} into the Volterra type integral equation:  
\begin{align}
F(x)=F_0(x)
&+\frac1k\int_0^x\sin k(x-y)P_+Q(y)F(y)\dd y
\notag
\\
&-\frac1k\int_0^x\sinh k(x-y)P_-Q(y)F(y)\dd y,
\label{eq:d2}
\end{align}
where $F_{0}(x)$ coincides with the right-hand side of~\eqref{eq:dd1}.
The Volterra equation \eqref{eq:d2} can be solved by the usual method of iterations, and one easily checks that any solution to \eqref{eq:d2} satisfies the differential equation \eqref{eq:d1}. 

The ensuing analysis of solutions of \eqref{eq:d1} is fairly standard and is probably present in the literature in various forms, so our presentation here will be brief. 
We need the following elementary proposition, which is an integral version of the Gronwall inequality, see e.g. \cite[Section IV.4]{BB}. It is sometimes called the Gronwall--Bellman lemma. 
\begin{lemma}\label{lma.d3}
Let $a=a(x)$ be a non-negative locally integrable function on $[0,\infty)$. Suppose that a non-negative function $h$ satisfies
\begin{equation}
h(x)\leq 1+\int_0^x a(y)h(y)\dd y
\label{d6}
\end{equation}
for all $x\geq0$. Then 
\begin{equation}
h(x)\leq \exp\left(\int_0^x a(y)\dd y\right)
\label{d7}
\end{equation}
for all $x\geq0$. 
\end{lemma}
\begin{proof}
Rewrite \eqref{d6} as
\[
\frac{h(x)}{1+\int_0^x a(y)h(y)\dd y}\leq 1.
\]
Upon multiplying by $a(x)$, we recognise a logarithmic derivative on the left.  Integrating over $x$, we obtain 
\[
\log \left(1+\int_0^{x}a(y)h(y)\dd y\right)\leq \int_0^x a(y)\dd y.
\]
Exponentiating and using \eqref{d6}, we arrive at \eqref{d7}.
\end{proof}

%%%%%%%%%%%%%%%%%%
\begin{lemma}\label{lma.d4}
%%%%%%%%%%%%%%%%%%
Let $k\in\bbC$, $k\not=0$ and $\kappa=\max\{\abs{\Re k},\abs{\Im k}\}$. 
\begin{enumerate}[\rm (i)]
\item
Let $F$ be a solution to \eqref{eq:d1} with $F'(0)=0$. 
Then $F$ satisfies the estimate
\begin{equation}
\norm{F(x)-F_0(x)}
\leq
2\norm{F(0)}\ee^{\kappa x}\left[\exp\left(\frac2{\abs{k}}\int_0^x\norm{Q(y)}\dd y\right)-1\right],
\label{d6a}
\end{equation}
where 
\[
F_0(x)=P_+F(0)\cos kx+P_-F(0)\cosh kx. 
\]
\item
Let $F$ be a solution to \eqref{eq:d1} with $F(0)=0$. 
Then $F$ satisfies the estimate
\begin{equation}
\norm{F(x)-F_0(x)}
\leq
\frac{2}{\abs{k}}\norm{F'(0)}\ee^{\kappa x}\left[\exp\left(\frac2{\abs{k}}\int_0^x\norm{Q(y)}\dd y\right)-1\right],
\label{d8}
\end{equation}
where
\[
F_0(x)=P_+F'(0)\frac1k\sin kx+P_-F'(0)\frac1k\sinh kx. 
\]
\end{enumerate}
\end{lemma}
\begin{proof}
Denote 
\[
G(x):=F(x)-F_0(x); 
\]
then we can rewrite the  integral equation \eqref{eq:d2} as
\begin{align}
G(x)=&G_0(x)
+\frac1k\int_0^x\sin k(x-y)P_+Q(y)G(y)\dd y
\notag
\\
&-\frac1k\int_0^x\sinh k(x-y)P_-Q(y)G(y)\dd y,
\label{d4}
\end{align}
where
\[
G_0(x):=\frac1k\int_0^x\sin k(x-y)P_+Q(y)F_0(y)\dd y
-\frac1k\int_0^x\sinh k(x-y)P_-Q(y)F_0(y)\dd y.
\]
Consider the case (i). We have the elementary estimates
\[
\abs{\sin kx}\leq \ee^{\kappa x}, \quad 
\abs{\sinh kx}\leq \ee^{\kappa x},\quad 
\norm{F_0(x)}\leq 2\norm{F(0)}\ee^{\kappa x}.
\]
It follows that 
\[
\norm{G_0(x)}\leq \frac{4}{\abs{k}}\norm{F(0)}\int_0^x \ee^{\kappa(x-y)} \norm{Q(y)}\ee^{\kappa y}\dd y
=\ee^{\kappa x}\frac{4}{\abs{k}}\norm{F(0)}\int_0^x  \norm{Q(y)}\dd y.
\]
From \eqref{d4} we obtain
\begin{equation}
\norm{G(x)}\leq \norm{G_0(x)}+\frac{2}{\abs{k}}\int_0^x \ee^{\kappa(x-y)}\norm{Q(y)}\norm{G(y)}\dd y. 
\label{d5}
\end{equation}
Denote for brevity 
\begin{align*}
g(x):=\norm{G(x)}\ee^{-\kappa x}, 
\qquad
q_+(x):=\frac{2}{\abs{k}}\int_0^x \norm{Q(y)}\dd y.
\end{align*}
Then \eqref{d5} rewrites as 
\[
g(x)\leq 2\norm{F(0)}q_+(x)+\int_0^x q_{+}'(y)g(y)\dd y.
\]
Without loss of generality, we may assume that $F(0)\neq0$, since otherwise $F\equiv0$ and the claim holds trivially. Let us divide the last inequality by $2\norm{F(0)}$ and add $1$ to both sides. We obtain
\[
h(x)\leq 1+\int_0^x q_+'(y)h(y)\dd y, \;\mbox{ for }\; h(x):=1+\frac{g(x)}{2\norm{F(0)}}. 
\]
By Lemma~\ref{lma.d3}, we find
\[
h(x)\leq \ee^{q_+(x)}, \quad \text{ or }\quad
g(x)\leq 2\norm{F(0)}\left(\ee^{q_+(x)}-1\right),
\]
which rewrites as \eqref{d6a}.

In case (ii), the argument is analogous. We indicate the changes: 
the estimates for $F_0$ and $G_0$ will be 
\[
\norm{F_0(x)}\leq \frac{2}{\abs{k}}\norm{F'(0)}\ee^{\kappa x}, 
\quad
\norm{G_0(x)}\leq 
\ee^{\kappa x}\frac{4}{\abs{k}^2}\norm{F'(0)}\int_0^x  \norm{Q(y)}\dd y.
\]
Then, with the same notation for $g$ and $q_+$, we obtain 
\[
g(x)\leq \frac{2}{\abs{k}}\norm{F'(0)}q_+(x)+\int_0^x q_+'(y)g(y)\dd y,
\]
which leads to
\[
g(x)\leq \frac{2}{\abs{k}}\norm{F'(0)}\left(\ee^{q_+(x)}-1\right),
\]
which is equivalent to \eqref{d8}.
\end{proof}

%%%%%%%%%%%%%%%%%%%%%%%%%%%%%%
\subsection{Proof of Theorem~\ref{thm:d1}}\label{subsec:proof-asympt-thm}

For notational convenience, we set $k=\kappa+\ii\kappa$ with $\kappa>0$. It will be convenient to start by considering the case of solutions corresponding to the Neumann boundary condition at $x=0$, i.e. $\alpha=0$; we denote these solutions by $\Phi_{0}$, $\Theta_{0}$. 

First consider $\Phi_{0}$.  By Lemma~\ref{lma.d4}(i), we have
\[
\norm{\Phi_{0}(x,\lambda)-P_{+}\cos kx-P_{-}\cosh kx}\leq \ee^{\kappa x}\,\calO_x(1/\kappa)
\]
for $\kappa\to\infty$. Taking also into account that 
\[
P_+\cos kx + P_-\cosh kx
=\frac12 \ee^{\kappa x}\left(\ee^{-\ii\kappa x}P_{+} +\ee^{\ii\kappa x}P_{-}\right)+\calO(\ee^{-\kappa x}),
\]
we find
\[
\left\|\Phi_{0}(x,\lambda)-\tfrac12 \ee^{\kappa x}\left(\ee^{-\ii\kappa x}P_+ +\ee^{\ii\kappa x}P_-\right)\right\|\leq \ee^{\kappa x}\,\calO_x(1/\kappa)
\]
as $\kappa\to\infty$. It follows that 
\begin{equation}
\Phi_{0}(x,\lambda)=\frac12\ee^{\kappa x}\left(\ee^{-\ii\kappa x}P_+ +\ee^{\ii\kappa x}P_-\right)\left(1+\calO_x\left(1/{\kappa}\right)\right),
\label{eq:d9}
\end{equation}
as $\kappa\to\infty$; here it is important that the matrix $(\ee^{-\ii\kappa x}P_+ -\ee^{\ii\kappa x}P_-)$
is invertible (in fact, unitary). In a similar way, for $\Theta_{0}$ we obtain 
\begin{equation}
\Theta_{0}(x,\lambda)=\frac{1}{2(\kappa+\ii\kappa)}\ee^{\kappa x}\left(\ii\ee^{-\ii\kappa x}P_+ -\ee^{\ii\kappa x}P_-\right)\left(1+\calO_x\left(1/{\kappa}\right)\right)
\label{eq:d10}
\end{equation}
as $\kappa\to\infty$. 
For a general parameter $\alpha\in\mathbb{C}\cup\{\infty\}$, one can express $\Phi$ and $\Theta$ in terms of $\Phi_{0}$ and $\Theta_{0}$ as 
\begin{equation}
\begin{aligned}
\Phi&=\Phi_{0} S-\Theta_{0} C,\\
\Theta&=\Phi_{0} C+\Theta_{0} S.
\end{aligned}
\label{eq:phi-theta-by-neumann}
\end{equation}
Substituting~\eqref{eq:d9} and~\eqref{eq:d10} into \eqref{eq:phi-theta-by-neumann} yields the required asymptotics. 
The proof of Theorem~\ref{thm:d1} is complete. 
\qed

%%%%%%%%%%%%%%%%%%%%%%%%%%%%%%
%%%%%%%%%%%%%%%%%%%%%%%%%%%%%%
\section{Uniqueness II: the main argument}
\label{sec.c}
%%%%%%%%%%%%%%%%%%%%%%%%%%%%%%
%%%%%%%%%%%%%%%%%%%%%%%%%%%%%

In this section we complete the proof of Theorem~\ref{thm:a1}. 

\subsection{Decay of solutions along a non-real ray}
We recall that the solution $X$ is defined by \eqref{eq:X}. The purpose of this subsection is to prove the following auxiliary statement. 
\begin{theorem}\label{thm.cc1}
Let the solutions $\Phi$ and $X$ be as defined in Section~\ref{sec:a3}. 
Then for any $a>0$ we have 
\begin{align}
\Phi(a,\lambda)X(a,\overline{\lambda})^*
=X(a,\lambda)\Phi(a,\overline{\lambda})^*&\to0, 
\label{eq:c1}
\end{align}
as $\lambda\to\infty$ along any non-real ray.
\end{theorem}
By Proposition~\ref{prop:resolvent_kernel} below, the expression in the left-hand side of \eqref{eq:c1} is minus the resolvent kernel of $\bH$ on the diagonal $x=y=a$, and so the statement should not be surprising. However, we prefer to give a direct argument in terms of the asymptotics of the $M$-function. The proof below is inspired by the technique of introducing additional boundary condition at $x=a$, see e.g. \cite[Eq.~(3.19)]{GSimon}.

First we need some notation. We fix some $\lambda\in\bbC\setminus\bbR$. For $a>0$, let $M_{*}(\lambda;a)$ be the $M$-function for the operator $\bH$ but considered on the interval $(a,\infty)$ with the Dirichlet boundary condition at $x=a$. Furthermore, let $\wt{M}_{*}(\lambda;a)$ be the $M$-function for the operator
\begin{equation}
-\epsilon\frac{\dd^2}{\dd x^2}+Q(a-x)
\label{eq:c6}
\end{equation}
(note $a-x$ in place of $x$) on the interval $(0,a)$  with the Dirichlet boundary condition at $x=0$ and the boundary condition 
\[
CF(a)-S\epsilon F'(a)=0
\]
at $x=a$ (note the sign change in front of $F'$ compared to \eqref{eq:a2a}).

We recall that by Lemma~\ref{lma.nonsingular}, for any $\lambda\in\bbC\setminus\bbR$, the matrices $\Phi(a,\lambda)$ and $X(a,\lambda)$ are non-singular. 

\begin{lemma}\label{lma.cc2}
For any $\lambda\in\bbC\setminus\bbR$, one has
\begin{align}
M_{*}(\lambda;a)&= X'(a,\lambda)X(a,\lambda)^{-1}\epsilon,
\label{eq:cc3}
\\
\wt{M}_{*}(\lambda;a)&=-\Phi'(a,\lambda)\Phi(a,\lambda)^{-1}\epsilon.
\label{eq:cc4}
\end{align}
\end{lemma}
\begin{proof}
In order to explain the idea of the proof, let us first consider our original spectral problem for the equation \eqref{eq:init_diff_eq}, but with the Dirichlet boundary condition at $x=0$, i.e. $\alpha=\infty$. Let $F$ be any $2\times2$ matrix solution to the equation \eqref{eq:init_diff_eq} which belongs to $L^2(\bbR_+)$. Suppose $F(0)$ is non-singular; then using the dimension lemma (Lemma~\ref{lem:dim_S}) and comparing initial conditions, we find that $X(x,\lambda)=F(x)F(0)^{-1}\epsilon$. It follows that 
\[
M_{\infty}(\lambda)= X'(0,\lambda)= F'(0)F(0)^{-1}\epsilon. 
\]
The previous argument can be applied to the problem on $(a,\infty)$ instead of $(0,\infty)$. This yields
\[
M_{*}(\lambda;a)= F'(a)F(a)^{-1}\epsilon,
\]
where $F$ is any square integrable solution with $F(a)$ non-singular. Now it suffices to take $F(x)=X(x,\lambda)$, and \eqref{eq:cc3} follows. 

Considering the problem for \eqref{eq:c6} on $(0,a)$ and applying the same logic, we obtain \eqref{eq:cc4}. One has to keep in mind the minus sign appearing in front of $F'$ because of the change of variable $x\mapsto a-x$. 
\end{proof}

\begin{lemma}\label{lma.cc1}
For any $\lambda\in\bbC\setminus\bbR$, we have 
\[
\Phi(a,\lambda)X(a,\overline{\lambda})^*
=X(a,\lambda)\Phi(a,\overline{\lambda})^*
=\epsilon\big(M_{*}(\lambda;a)+\wt{M}_{*}(\lambda;a)\big)^{-1}\epsilon.
\]
\end{lemma}

\begin{proof}
Let us use the Wronskian
\[
W:=[X(\cdot,\overline{\lambda})^*,\Phi(\cdot,\lambda)]=
X(x,\overline{\lambda})^*\epsilon\Phi'(x,\lambda)-X'(x,\overline{\lambda})^*\epsilon\Phi(x,\lambda).
\]
It follows from definition~\eqref{eq:X} and formulas~\eqref{eq:wronsk_1} and \eqref{eq:wronsk_2} that $W=-I$. Now we write
\[
\Phi(a,\lambda)X(a,\overline{\lambda})^*
=
-\Phi(a,\lambda)W^{-1}X(a,\overline{\lambda})^*
=
-\left((X(a,\overline{\lambda})^*)^{-1}W\Phi(a,\lambda)^{-1}\right)^{-1},
\]
where $\Phi(a,\lambda)$ and $X(a,\overline{\lambda})$ are invertible by Lemma~\ref{lma.nonsingular}. By the definition of $W$ (substituting $x=a$) the matrix in the right-hand side can be transformed as
\[
(X(a,\overline{\lambda})^*)^{-1}W\Phi(a,\lambda)^{-1}
=
\epsilon\Phi'(a,\lambda)\Phi(a,\lambda)^{-1}
-
(X(a,\overline{\lambda})^*)^{-1}X'(x,\overline{\lambda})^*\epsilon.
\]
Using Lemma~\ref{lma.cc2} and symmetry \eqref{eq:a4}, this rewrites as
\[
(X(a,\overline{\lambda})^*)^{-1}W\Phi(a,\lambda)^{-1}
=
-\epsilon\big(\wt{M}_{*}(\lambda;a)+M_{*}(\lambda;a)\big)\epsilon.
\]
This proves the identity
\[
\Phi(a,\lambda)X(a,\overline{\lambda})^*
=\epsilon\big(M_{*}(\lambda;a)+\wt{M}_{*}(\lambda;a)\big)^{-1}\epsilon.
\]
By writing $\overline{\lambda}$ instead of $\lambda$ in the last equality, and taking adjoints, one gets
\[
X(a,\lambda)\Phi(a,\overline{\lambda})^*=\epsilon\big(M_{*}(\overline{\lambda};a)^{*}+\wt{M}_{*}(\overline{\lambda};a)^{*}\big)^{-1}\epsilon=\epsilon\big(M_{*}(\lambda;a)+\wt{M}_{*}(\lambda;a)\big)^{-1}\epsilon.
\]
The proof is complete. 
\end{proof}

\begin{proof}[Proof of Theorem~\ref{thm.cc1}]
It suffices to consider $\abs{\lambda}\to\infty$ along a ray in the upper half-plane. Then we can take $\lambda=k^2$ with $\Re k>0$ and $\Im k>0$. By \eqref{eq.Masymp2}, see also Remark~\ref{rem:M_asympt_compact}, we have
\[
M_{*}(\lambda;a)+\wt{M}_{*}(\lambda;a)
=2k(\ii P_++P_-)+\smallO(k).
\]
Now we make use of Lemma~\ref{lma.cc1} and the invertibility of the matrix $\ii P_{+}+P_{-}$ to complete the proof.
\end{proof}

\subsection{The boundary parameter $\alpha$ is uniquely determined}
We begin the proof of Theorem~\ref{thm:a1}. 
Let $\bH$ (resp. $\widehat{\bH}$) be the operator with a potential $Q$ (resp. $\widehat{Q}$) and boundary parameter $\alpha$ (resp. $\widehat{\alpha}$). We will use hats for various quantities associated with $\widehat{\bH}$, such as $\widehat\Phi$, $\widehat X$, $\widehat A$, $\widehat{M}_{\widehat{\alpha}}(\lambda)$ etc. We assume that the spectral measures of $\bH$ and $\widehat{\bH}$ coincide, i.e. $\widehat{\Sigma}=\Sigma$. 
Our first aim is to show that $\widehat{\alpha}=\alpha$ and $\widehat{M}_{\widehat{\alpha}}(\lambda)=M_\alpha(\lambda)$. 
\begin{lemma}
If $\widehat{\Sigma}=\Sigma$, then  $\widehat{\alpha}=\alpha$ and
$\widehat{M}_{\widehat{\alpha}}(\lambda)=M_\alpha(\lambda)$
for all $\lambda\in\bbC\setminus\bbR$.
\label{lem:Sasha-great-lemma}
\end{lemma}
\begin{proof}
We first observe that from the high energy asymptotics of the measure $\Sigma$ in Theorem~\ref{thm:asympt_Sigma} we immediately read off $\abs{\alpha}$. Thus, we necessarily have $\abs{\alpha}=\abs{\widehat\alpha}$. We need to work a little harder to prove that the arguments of the complex numbers $\alpha$ and $\widehat{\alpha}$ coincide, too.

We start by observing that for $\alpha\not=\infty$, by virtue of Theorem~\ref{thm:asympt_Sigma}, the integral
\begin{equation}
\int_{-\infty}^\infty \frac{\dd\Sigma(t)}{1+\abs{t}}
\label{eq:x5}
\end{equation}
converges. It follows from~\eqref{eq.Masymp1} and the dominated convergence that the representation \eqref{eq:intres} simplifies to 
\begin{equation}
M_\alpha(\lambda)=
\epsilon A+\int_{-\infty}^\infty \frac{\dd\Sigma(t)}{t-\lambda}, \quad 
\lambda\in\bbC\setminus\bbR.
\label{eq:intres1}
\end{equation}
We consider separately three cases. 

\emph{Case 1: $\abs{\widehat{\alpha}}=\abs{\alpha}=0$.} This case can be subsumed into Case 2, but for clarity we consider it separately. We already have $\widehat{\alpha}=\alpha=0$, and so $\widehat{A}=A=0$. By \eqref{eq:intres1}, here we have $\widehat{M}_0=M_0$, and we are done. 

\emph{Case 2: $\abs{\widehat{\alpha}}=\abs{\alpha}$ is finite and non-zero. } Denote 
\[
N(\lambda)=\int_{-\infty}^\infty \frac{\dd\Sigma(t)}{t-\lambda}, \quad \lambda\in\bbC\setminus\bbR.
\]
Since $\Sigma=\widehat{\Sigma}$, by \eqref{eq:intres1} we have
\[
M_\alpha(\lambda)=\epsilon A+N(\lambda) 
\quad\text{ and }\quad 
\widehat{M}_{\widehat{\alpha}}(\lambda)=\epsilon\widehat{A}+N(\lambda),
\]
and therefore 
\[
M_\alpha(\lambda)-\widehat{M}_{\widehat{\alpha}}(\lambda)=\epsilon A-\epsilon\widehat{A},
\]
where the right hand side is constant. Now let us apply the asymptotic formula \eqref{eq.Masymp1} to both $M_\alpha$ and $\widehat{M}_{\widehat{\alpha}}$. Inspecting the terms of the order $1/k^2$ and taking into account the explicit expression \eqref{eq.Masymp4} for the matrix $(\ii P_{+}-P_{-})\epsilon A(\ii P_{+}-P_{-})$, we find that $\alpha=\widehat{\alpha}$. Now coming back to \eqref{eq:intres1}, we finally conclude that $M_\alpha=\widehat{M}_{\widehat{\alpha}}$. 

\emph{Case 3: $\widehat{\alpha}=\alpha=\infty$.}
Subtracting the representations \eqref{eq:intres} for $M_\infty$ and $\widehat{M}_{\infty}$, we find
\[
\widehat{M}_{\infty}(\lambda)-M_\infty(\lambda)=\widehat{M}_{\infty}(\ii)-M_\infty(\ii),
\]
where the right hand side is constant. Now using the asymptotics \eqref{eq.Masymp2}, we see that this constant must vanish. From here we find that $M_\infty=\widehat{M}_{\infty}$.
\end{proof}

In the rest of this section we prove that $Q=\widehat{Q}$. 

\subsection{The use of Phragm{\' e}n--Lindel{\" o}f principle}
\begin{lemma}
For all $\lambda\in\bbC$ and all $x\geq0$ we have 
\begin{equation}
\Theta(x,\lambda)\widehat{\Phi}(x,\overline{\lambda})^*
=
\Phi(x,\lambda)\widehat{\Theta}(x,\overline{\lambda})^*.
\label{eq:c4}
\end{equation}
\end{lemma}

\begin{proof}
Let us write \eqref{eq:c1}, with and without hats, as follows:
\begin{equation}
X(x,\lambda)\Phi(x,\overline{\lambda})^*\to0
\quad\mbox{ and }\quad
\widehat{\Phi}(x,\lambda)\widehat{X}(x,\overline{\lambda})^*\to0,
\label{eq:dd4}
\end{equation}
as $\lambda\to\infty$ along any non-real ray. 

Let us take $\lambda=\ii\kappa$  with $\kappa>0$. By Theorem~\ref{thm:d1}, taking into account that the leading terms in \eqref{eq:dd2}  are invertible matrices independent on the potentials, we find the limit
\[
\widehat{\Phi}(x,\ii\kappa)\Phi(x,\ii\kappa)^{-1}\to I,
\]
as $\kappa\to\infty$. 
In the same way (or by using the symmetries \eqref{eq:a7} below) one can check that the same relation holds as $\kappa\to-\infty$. 
Applying this to \eqref{eq:dd4}, we find 
\[
X(x,\lambda)\widehat{\Phi}(x,\overline{\lambda})^*\to0
\quad\mbox{ and }\quad
\Phi(x,\lambda)\widehat{X}(x,\overline{\lambda})^*\to0,
\]
with $\lambda=\pm \ii\kappa$ and $\kappa\to\infty$.
Recalling the definition~\eqref{eq:X} of $X$ and $\widehat{X}$, this rewrites as
\begin{align*}
\left(\Theta(x,\lambda)-\Phi(x,\lambda)M_\alpha(\lambda)\right)\widehat{\Phi}(x,\overline{\lambda})^*&\to0, 
\\
\Phi(x,\lambda)(\widehat\Theta(x,\overline{\lambda})-\widehat{\Phi}(x,\overline{\lambda})\widehat{M}_{\widehat{\alpha}}(\overline{\lambda}))^*&\to0. 
\end{align*}
Subtracting the last two identities and using that $\widehat{M}_{\widehat{\alpha}}=M_\alpha$, see Lemma~\ref{lem:Sasha-great-lemma}, together with~\eqref{eq:a4}, we find that the terms with the $M$-function cancel out, and we obtain
\begin{equation}
\Theta(x,\lambda)\widehat{\Phi}(x,\overline{\lambda})^*
-
\Phi(x,\lambda)\widehat{\Theta}(x,\overline{\lambda})^*
\to0
\label{eq:c3}
\end{equation}
with $\lambda=\pm\ii\kappa$ and $\kappa\to\infty$. Observe that the left-hand side of \eqref{eq:c3} is an entire function in $\lambda$ which is of order $1/2$, as it follows from Lemma~\ref{lma.d4}. Now we can apply the Phragm{\' e}n--Lindel{\" o}f principle separately to the right half-plane $\Re \lambda\geq0$ and to the left half-plane $\Re \lambda\leq0$ (regarding them as sectors with opening $\pi$). This shows that our entire function is bounded in both half-planes and therefore, by the Liouville theorem, it is constant. Using \eqref{eq:c3} once again, we conclude that our function vanishes identically. This yields the required identity \eqref{eq:c4}.
\end{proof}

\subsection{Concluding the argument}
The ensuing algebra is less obvious than in the scalar case. 
We will mainly follow the argument of \cite{And} (which is a matrix generalisation of \cite{B}). In what follows, in the proofs we will take $\lambda$ non-real and freely use the inverses of $\Phi(x,\lambda)$ and $\widehat{\Phi}(x,\lambda)$; they exist by Lemma~\ref{lma.nonsingular}. After the required identities have been proved for non-real $\lambda$, they extend to real $\lambda$ by analyticity.

\begin{lemma}
For all $\lambda\in\bbC$ and all $x\geq0$ we have
\begin{equation}
\Phi(x,\overline{\lambda})^{*}\epsilon\Phi(x,\lambda)=\widehat{\Phi}(x,\overline{\lambda})^{*}\epsilon\widehat{\Phi}(x,\lambda).
\label{eq:c7}
\end{equation}
\end{lemma}
\begin{proof}
As discussed, initially we take $\Im \lambda\not=0$. Bearing in mind the initial conditions~\eqref{eq:phi,theta_bc_cond}, it is also sufficient to assume $x>0$.

The first vanishing Wronskian from~\eqref{eq:wronsk_1} rewrites as 
\[
\epsilon\Phi'(x,\lambda)\Phi(x,\lambda)^{-1}=\left(\Phi(x,\overline{\lambda})^{*}\right)^{-1}\Phi'(x,\overline{\lambda})^{*}\epsilon.
\]
Applying this identity in the following computation, we find 
\begin{align*}
\Phi(x,\overline{\lambda})^{*}\epsilon\Phi(x,\lambda)&\left(\Phi(x,\lambda)^{-1}\Theta(x,\lambda)\right)'\\
&=-\Phi(x,\overline{\lambda})^{*}\epsilon\Phi'(x,\lambda)\Phi(x,\lambda)^{-1}\Theta(x,\lambda)+\Phi(x,\overline{\lambda})^{*}\epsilon\Theta'(x,\lambda)\\
&=-\Phi'(x,\overline{\lambda})^{*}\epsilon\Theta(x,\lambda)+\Phi(x,\overline{\lambda})^{*}\epsilon\Theta'(x,\lambda).
\end{align*}
Using also~\eqref{eq:wronsk_2}, we obtain the identity 
\begin{equation}
\Phi(x,\overline{\lambda})^{*}\epsilon\Phi(x,\lambda)\left(\Phi(x,\lambda)^{-1}\Theta(x,\lambda)\right)'=I.
\label{eq:expr_id1}
\end{equation}
Then we also have
\[
\left(\Phi(x,\lambda)^{-1}\Theta(x,\lambda)\right)'\Phi(x,\overline{\lambda})^{*}\epsilon\Phi(x,\lambda)=I.
\]
Of course, the same identity is true for the solutions with hats. Replacing $\lambda$ by $\overline{\lambda}$ and taking adjoints in this identity yields
\begin{equation}
\widehat{\Phi}(x,\overline{\lambda})^{*}\epsilon\widehat{\Phi}(x,\lambda)\left(\widehat{\Theta}(x,\overline{\lambda})^{*}(\widehat{\Phi}(x,\overline{\lambda})^{*})^{-1}\right)'=I.
\label{eq:expr_id2}
\end{equation}

Now we return to \eqref{eq:c4} and write it as
\[
\Phi(x,\lambda)^{-1}\Theta(x,\lambda)=\widehat{\Theta}(x,\overline{\lambda})^*\big(\widehat{\Phi}(x,\overline{\lambda})^{*}\big)^{-1}.
\]
When combined with~\eqref{eq:expr_id1} and~\eqref{eq:expr_id2}, we deduce~\eqref{eq:c7}.
\end{proof}

%%%%%%%%%%%%%%%%%%
\begin{lemma}
%%%%%%%%%%%%%%%%%%
For all $\lambda\in\bbC$ and all $x\geq0$ one has
\begin{equation}
\Phi(x,\lambda)=\widehat{\Phi}(x,\lambda).
\label{eq:c13}
\end{equation}
\end{lemma}
\begin{proof}
Again, it suffices to assume $\lambda$ non-real and $x>0$. 
Let us differentiate \eqref{eq:c7}:
\[
\Phi'(x,\overline{\lambda})^{*}\epsilon\Phi(x,\lambda)+\Phi(x,\overline{\lambda})^{*}\epsilon\Phi'(x,\lambda)
=
\widehat{\Phi}'(x,\overline{\lambda})^{*}\epsilon\widehat{\Phi}(x,\lambda)+\widehat{\Phi}(x,\overline{\lambda})^{*}\epsilon\widehat{\Phi}'(x,\lambda)
\]
and use \eqref{eq:wronsk_1} on both sides; this gives
\[
\Phi(x,\overline{\lambda})^{*}\epsilon\Phi'(x,\lambda)
=
\widehat{\Phi}(x,\overline{\lambda})^{*}\epsilon\widehat{\Phi}'(x,\lambda).
\]
Taking the inverses in~\eqref{eq:c7} and multiplying on the right by both sides of the last identity, we find
\begin{equation}
\Phi(x,\lambda)^{-1}\Phi'(x,\lambda)=\widehat{\Phi}(x,\lambda)^{-1}\widehat{\Phi}'(x,\lambda).
\label{eq:c12}
\end{equation}
For the rest of the proof we drop $\lambda$ from our notation. 
We  differentiate $\widehat{\Phi}\Phi^{-1}$ and use \eqref{eq:c12} to express $\widehat{\Phi}'$:
\begin{align*}
(\widehat{\Phi}(x)\Phi^{-1}(x))'
&=
\widehat{\Phi}'(x)\Phi(x)^{-1}-\widehat{\Phi}(x)\Phi(x)^{-1}\Phi'(x)\Phi(x)^{-1}
\\
&=\widehat{\Phi}(x)\Phi(x)^{-1}\Phi'(x)\Phi(x)^{-1}- \widehat{\Phi}(x)\Phi(x)^{-1}\Phi'(x)\Phi(x)^{-1}=0.
\end{align*}
We conclude that the matrix $\widehat{\Phi}(x)\Phi^{-1}(x)$ is independent of $x$; let us write this as
\begin{equation}
\widehat{\Phi}(x)=D(\lambda)\Phi(x),
\label{eq:Phi}
\end{equation}
where $D(\lambda)$ is a matrix that may depend on $\lambda$ but is independent of $x$. If $\alpha\neq\infty$, then using the initial conditions~\eqref{eq:phi,theta_bc_cond}, we find that $D(\lambda)=I$, which completes the proof. If $\alpha=\infty$, then we differentiate \eqref{eq:Phi} with respect to $x$ and again compare with the initial conditions; this again shows that $D(\lambda)=I$. 
\end{proof}

Now we can conclude. 

\begin{proof}[Proof of Theorem~\ref{thm:a1}]
Differentiating~\eqref{eq:c13} twice and using the differential equations for $\Phi$ and $\widehat{\Phi}$, we obtain $Q=\widehat{Q}$. 
\end{proof}

%%%%%%%%%%%%%%%%%%%%%%%%%%%%%%
%%%%%%%%%%%%%%%%%%%%%%%%%%%%%%
\section{Proofs of Theorems~\ref{thm:a2} and~\ref{thm.diagh}}
\label{sec.b}
%%%%%%%%%%%%%%%%%%%%%%%%%%%%%%
%%%%%%%%%%%%%%%%%%%%%%%%%%%%%

\subsection{Proof of Theorem~\ref{thm:a2}}\label{sec.b3a}
\mbox{}

\emph{Step 1:} We start by establishing two symmetry conditions for $M_\alpha(\lambda)$. 
Observe that, due to the special structure of $Q(x)$ from \eqref{eq:a2}, we have
\[
\epsilon\overline{Q(x)}\epsilon=Q(x).
\]
From here it follows that if $F$ satisfies the eigenvalue equation \eqref{eq:init_diff_eq}, then $\epsilon \overline{F}$ satisfies the same equation with $\overline{\lambda}$ in place of $\lambda$. Furthermore, any solution to \eqref{eq:init_diff_eq} can be multiplied on the right by a constant $2\times 2$ matrix; hence $\epsilon\overline{F}\epsilon$ is also a solution of \eqref{eq:init_diff_eq} with $\overline{\lambda}$ in place of $\lambda$. 

Now consider the functions
\[
\widetilde{\Theta}(x,\lambda):=\epsilon\overline{\Theta(x,\overline{\lambda})}\epsilon, 
\quad
\widetilde{\Phi}(x,\lambda):=\epsilon\overline{\Phi(x,\overline{\lambda})}\epsilon.
\]
By the above reasoning, they are solutions of \eqref{eq:init_diff_eq} and by inspection they satisfy the same initial conditions \eqref{eq:phi,theta_bc_cond} as $\Theta$ and $\Phi$. It follows that 
\begin{equation}
\widetilde\Theta=\Theta, \quad 
\widetilde\Phi=\Phi,
\label{eq:a7}
\end{equation}
and so we find
\begin{equation}
\widetilde{\Theta}(\cdot,\lambda)-\widetilde{\Phi}(\cdot,\lambda)M_\alpha(\lambda)\in L^2(\bbR_+).
\label{eq:a3}
\end{equation}
On the other hand, taking complex conjugation in 
\[
\Theta(\cdot,\overline{\lambda})-\Phi(\cdot,\overline{\lambda})M_\alpha(\overline{\lambda})\in L^2(\bbR_+)
\]
and multiplying it by $\epsilon$ on the right and on the left, we arrive at
\[
\widetilde{\Theta}(\cdot,\lambda)-\widetilde{\Phi}(\cdot,\lambda)(\epsilon \overline{M_\alpha(\overline{\lambda})}\epsilon)\in L^2(\bbR_+).
\]
Comparing with \eqref{eq:a3} and using the uniqueness of the definition of $M_\alpha(\lambda)$, we find
\[
M_\alpha(\lambda)=\epsilon\overline{M_\alpha(\overline{\lambda})}\epsilon.
\]
Finally, using the relation~\eqref{eq:a4}, we arrive at
\[
M_\alpha(\lambda)=\epsilon M_\alpha(\lambda)^T\epsilon,
\]
where $M_\alpha^T$ is the matrix transpose of $M_\alpha$. Moreover, the constant term $\Re M_\alpha(\ii)$ in the integral representation \eqref{eq:intres} enjoys the same symmetry and therefore 
\begin{equation}
M_\alpha(\lambda)-\Re M_\alpha(\ii)=\epsilon (M_\alpha(\lambda)-\Re M_\alpha(\ii))^T\epsilon.
\label{eq:a5}
\end{equation}
This is our first symmetry condition on $M_\alpha(\lambda)$. 

\emph{Step 2:}
Denote 
\[
\xi:=\begin{pmatrix}1&0\\0&-1\end{pmatrix}
\]
and observe that
\[
\xi\epsilon\xi=-\epsilon \quad\text{ and }\quad \xi Q(x)\xi=-Q(x).
\]
It follows that if $F$ satisfies the eigenvalue equation \eqref{eq:init_diff_eq}, then $\xi F\xi$ satisfies the same equation with $-\lambda$ in place of $\lambda$. By inspecting the initial conditions \eqref{eq:phi,theta_bc_cond}, we see that 
\[
\xi\Phi(x,-\lambda)\xi=\Phi(x,\lambda), \quad
\xi\Theta(x,-\lambda)\xi=-\Theta(x,\lambda)
\]
(note the minus in the right-hand side of the second equation).
Now arguing in exactly the same way as in the previous step, we obtain the second symmetry relation 
\[
M_\alpha(\lambda)=-\xi M_\alpha(-\lambda)\xi
\]
for the $M$-function. The constant term $\Re M_\alpha(\ii)$ satisfies the same symmetry, and therefore 
\begin{equation}
M_\alpha(\lambda)-\Re M_\alpha(\ii)=-\xi (M_\alpha(-\lambda)-\Re M_\alpha(\ii))\xi.
\label{eq:a6}
\end{equation}

\emph{Step 3:} Now we translate the symmetry relations \eqref{eq:a5} and \eqref{eq:a6} for the $M$-function into symmetry relations for the measure $\Sigma$. It will be convenient to represent the spectral measure $\Sigma$ as
\[
\dd\Sigma(\lambda)=
R(\lambda)\dd\nu(\lambda),
\]
where $\nu$ is the scalar measure defined by $\dd\nu:=\dd\Tr\Sigma/2$, and $R(\lambda)$ is a positive semi-definite matrix with trace equal to $2$, defined $\nu$-a.e.; the matrix $R(\lambda)$ can be understood as the Radon--Nikodym derivative $\dd\Sigma/\dd\nu$. 

Let us start with \eqref{eq:a5}. By the uniqueness of the measure in the integral representation \eqref{eq:intres} it implies that 
\[
R(\lambda)=\epsilon R(\lambda)^T\epsilon,
\]
i.e. in terms of the matrix elements of $R(\lambda)$,
\[
R_{1,1}(\lambda)=R_{2,2}(\lambda).
\]
Since $R$ has trace two,  we conclude that 
\[
R_{1,1}(\lambda)=R_{2,2}(\lambda)=1.
\]
Since the matrix $R(\lambda)$ is positive semi-definite, we find that it has the form 
\[
R(\lambda)=
\begin{pmatrix}
1&\psi(\lambda)
\\
\overline{\psi(\lambda)}&1
\end{pmatrix}
\]
with some complex-valued $\psi\in L^\infty(\nu)$ satisfying $\abs{\psi(\lambda)}\leq1$ for $\nu$-a.e. $\lambda\in\bbR$.

Now consider \eqref{eq:a6}; it gives
\begin{align*}
\int_{-\infty}^\infty\left(\frac{1}{t-\lambda}-\frac{t}{1+t^2}\right)R(t)\dd\nu(t)
&=
-\int_{-\infty}^\infty\left(\frac{1}{t+\lambda}-\frac{t}{1+t^2}\right)(\xi R(t)\xi)\dd\nu(t)
\\
&=
\int_{-\infty}^\infty\left(\frac{1}{t-\lambda}-\frac{t}{1+t^2}\right)(\xi R(-t)\xi)\dd\nu(-t).
\end{align*}
Again, by the uniqueness of the measure in the Herglotz--Nevanlinna integral representation, we have
\[
R(t)\dd\nu(t)=(\xi R(-t)\xi)\dd\nu(-t).
\]
Since
\[
\xi R(-t)\xi=\begin{pmatrix}1&-\psi(-t)\\-\overline{\psi(-t)}&1\end{pmatrix}, 
\]
this implies that $\nu$ is even and $\psi$ is odd. 
\qed

\subsection{Proof of Theorem~\ref{thm.diagh}}\label{sec.b3}
If $\lambda$ is an eigenvalue of $H$, then $\lambda$ has the geometric multiplicity $\leq1$ because there cannot be more than one linearly independent solution to the eigenvalue equation $-f''+qf=\lambda f$ satisfying the boundary condition \eqref{bc0}. In particular, the kernel of $H$ is either trivial or one-dimensional. Moreover, the operator $f\mapsto\overline f$ of complex conjugation in $L^2(\bbR_+)$ effects an anti-unitary equivalence of $H$ and $H^*$ and therefore the kernel of $H^*$ has the same dimension as $H$. 

It is a simple exercise in operator theory to check that the operators
\[
\bH=
\begin{pmatrix}
0&H\\ H^*&0
\end{pmatrix}
\quad\text{ and }\quad 
\wt\bH=
\begin{pmatrix}
\abs{H}&0\\0&-\abs{H^*}
\end{pmatrix}
\]
are unitarily equivalent (for any closed operator $H$). By the diagonalisation of $\bH$, see Proposition~\ref{thm.sp.decomp}, we conclude that $\wt\bH$ is unitarily equivalent to the operator $\calM$ of multliplication by the independent variable in $L_{\Sigma}^{2}(\bbR;\bbC)$. It follows that the operators
\[
\wt\bH_+:=\wt\bH\chi_{\bbR_{+}}(\wt\bH) 
\quad\text{ and }\quad 
\calM_+:=\calM\chi_{\bbR_{+}}(\calM),
\]
where $\chi_{\bbR_{+}}$ is the indicator function of set ${\bbR_{+}}\equiv(0,\infty)$, are unitarily equivalent. Therefore, they remain unitarily equivalent after restriction to the orthogonal complements to their kernels:
\[
\wt\bH_+|_{(\Ker \wt\bH_+)^\perp}
\quad\text{ is unitarily equivalent to }\quad
\calM_+|_{(\Ker\calM_+)^\perp}.
\]
Here the first operator equals $\abs{H}|_{(\Ker \abs{H})^\perp}$ while the second operator is $\calM|_{L^{2}_{\Sigma}(\bbR_{+};\bbC^{2})}$.

Finally, we prove \eqref{eq.multone} and \eqref{eq.multtwo}. By the already proven part of the theorem, the question reduces to the spectral multiplicity of the operator $\calM$ on $L_{\Sigma}^{2}(\bbR_{+};\bbC)$. Now the proof follows from the simple observation that the rank of the matrix 
\[
R(\lambda)=
\begin{pmatrix}
1&\psi(\lambda)
\\
\overline{\psi(\lambda)}&1
\end{pmatrix}
\]
is one on $S_1$ and two on $S_2$. 
The proof of Theorem~\ref{thm.diagh} is complete. 
\qed

%%%%%%%%%%%%%%%%%%%%%%%%%%%%%%%%%%%
%%%%%%%%%%%%%%%%%%%%%%%%%%%%%%%%%%%
\section{The self-adjoint and normal $H$}\label{sec.r}
%%%%%%%%%%%%%%%%%%%%%%%%%%%%%%%%%%%
%%%%%%%%%%%%%%%%%%%%%%%%%%%%%%%%%%%

Here we prove Theorems~\ref{thm:a5},  \ref{thm:a6} and \ref{thm:a7}. 

\subsection{Proof of Theorem~\ref{thm:a5}}
In this proof, it will be convenient to indicate the dependence of various quantities on $q$ and $\alpha$ explicitly, e.g. $\Phi(x,\lambda;q,\alpha)$, $\Sigma(s;q,\alpha)$, etc. 

We start with the observation that, by taking complex conjugation of the eigenvalue equation and checking the boundary conditions at $x=0$, we find
\[
\overline{\Phi(x,\lambda;q,\alpha)}=\Phi(x,\overline{\lambda};\overline{q},\overline{\alpha})
\quad\mbox{ and }\quad 
\overline{\Theta(x,\lambda;q,\alpha)}=\Theta(x,\overline{\lambda};\overline{q},\overline{\alpha}).
\]
From here and the definition of the $M$-function it is easy to conclude that 
\[
\overline{M_{\alpha}(\lambda;q)}=M_{\overline{\alpha}}(\overline{\lambda};\overline{q}).
\]
Using the uniqueness of the measure in the Herglotz--Nevanlinna integral representation, we find
\[
\dd\overline{\Sigma(s;q,\alpha)}=\dd\Sigma(s;\overline{q},\overline{\alpha})
\]
for all $s\in\bbR$, and so
\[
\dd\nu(s;q,\alpha)=\dd\nu(s;\overline{q},\overline{\alpha})\quad\text{ and }\quad \overline{\psi(s;q,\alpha)}=\psi(s;\overline{q},\overline{\alpha})
\]
for $\nu$-a.e. $s\in\bbR$ by Theorem~\ref{thm:a2}. 

Suppose $H$ is self-adjoint, i.e. $q$ is real-valued and $\alpha\in\bbR\cup\{\infty\}$. Then we get $ \overline{\psi(s;q,\alpha)}=\psi(s;q,\alpha)$, i.e. $\psi(s;q,\alpha)$ is real-valued. Conversely, suppose $\psi(s;q,\alpha)$ is real for $\nu$-a.e. $s\in\bbR$. It follows that 
\[
\dd\nu(s;q,\alpha)=\dd\nu(s;\overline{q},\overline{\alpha})
\quad\mbox{ and }\quad
\psi(s;q,\alpha)=\psi(s;\overline{q},\overline{\alpha}).
\]
Then by the uniqueness Theorem~\ref{thm:a3}, we conclude that $\overline{q}=q$ and $\alpha=\overline{\alpha}$, i.e. $q$ is real-valued and $\alpha\in\bbR\cup\{\infty\}$. 
The proof of Theorem~\ref{thm:a5} is complete. \qed

\subsection{Proof of Theorem~\ref{thm:a7}}
Throughout this subsection, $q$ is a bounded \emph{real} function, $\alpha\in\bbR\cup\{\infty\}$, and $H$ is the corresponding \emph{self-adjoint} operator. We use the objects introduced in Section~\ref{sec.warmup}: the scalar solutions $\varphi$, $\theta$, $\chi$ of the eigenvalue equation, the scalar $m$-function $m_\alpha(\lambda)$ and the scalar spectral measure $\sigma$.  Below $M_\alpha$ is the $2\times 2$ matrix-valued $M$-function corresponding to the potential 
\begin{equation}
Q=\begin{pmatrix}0&q+\ii\omega\\q-\ii\omega&0\end{pmatrix}.
\label{eq:X6}
\end{equation}
We start by relating $M_{\alpha}$ with $m_{\alpha}$.
\begin{lemma}
Let $\lambda$, $s$ be non-real complex numbers such that $s^2=\lambda^2+\omega^2$. Then 
\begin{equation}
M_\alpha(s)
=
\begin{pmatrix}
s&\lambda+\ii\omega\\ 
\lambda-\ii\omega& s
\end{pmatrix}
\frac{m_\alpha(\lambda)}{2\lambda}
-
\begin{pmatrix}
s&-\lambda+\ii\omega\\ 
-\lambda-\ii\omega& s
\end{pmatrix}
\frac{m_\alpha(-\lambda)}{2\lambda}.
\label{eq:X9}
\end{equation}
\end{lemma}
\begin{proof}
Let us denote 
\[
V_+:=\frac1{2\lambda}\begin{pmatrix}
s&\lambda+\ii\omega\\ 
\lambda-\ii\omega& s
\end{pmatrix}, 
\quad 
V_-:=-\frac1{2\lambda}
\begin{pmatrix}
s&-\lambda+\ii\omega\\ 
-\lambda-\ii\omega& s
\end{pmatrix}.
\]
Note that $V_-$ is obtained from $V_+$ by the change $\lambda\mapsto -\lambda$. 
We start with an elementary observation. Suppose $f$ satisfies the scalar eigenvalue equation
\[
-f''+qf=\lambda f.
\]
Then the $2\times 2$ matrix-valued function 
\[
F(x)=
V_+f(x)
\]
satisfies the equation 
\begin{equation}
-\epsilon F''+QF=sF,
\label{eq:X8}
\end{equation}
with $Q$ as in \eqref{eq:X6}. We will also use this observation with $\lambda$ replaced by $-\lambda$ throughout. 

Now let $\varphi$, $\theta$ be the solutions \eqref{eq:X11}. By our observation, 
\begin{align*}
\Phi(x,s)&=V_+\epsilon \varphi(x,\lambda)+V_-\epsilon\varphi(x,-\lambda),
\\
\Theta(x,s)&=V_+\theta(x,\lambda)+V_-\theta(x,-\lambda)
\end{align*}
are solutions to \eqref{eq:X8}. (Recall that any solution of \eqref{eq:X8} can be multiplied by a constant matrix on the right, hence we have multiplied by $\epsilon$ on the right in the first equation above.) By using formulas \eqref{eq:def_sc_scalar} and \eqref{eq:X11} with $\alpha\in\bbR\cup\{\infty\}$, we verify that $\Phi$, $\Theta$ satisfy the initial conditions \eqref{eq:phi,theta_bc_cond}. Now let us check that the right-hand side of \eqref{eq:X9} satisfies the definition of the $M$-function. In our matrix notation, the right-hand side of \eqref{eq:X9} reads 
\[
V_+m_\alpha(\lambda)+V_-m_\alpha(-\lambda).
\]
Using the matrix identities 
\[
V_+\epsilon V_-=V_-\epsilon V_+=0, \quad
V_+\epsilon V_+=V_+, \quad 
V_-\epsilon V_-=V_-, 
\]
we find
\begin{align*}
\Theta&(x,s)-\Phi(x,s)(V_+m_\alpha(\lambda)+V_-m_\alpha(-\lambda))
\\
=&V_+\theta(x,\lambda)+V_-\theta(x,-\lambda)
-
(V_+\epsilon \varphi(x,\lambda)+V_-\epsilon\varphi(x,-\lambda))
(V_+m_\alpha(\lambda)+V_-m_\alpha(-\lambda))
\\
=&V_+\theta(x,\lambda)+V_-\theta(x,-\lambda)
-
(V_+ \varphi(x,\lambda)m_\alpha(\lambda)+V_-\varphi(x,-\lambda)m_\alpha(-\lambda))
\\
=&V_+\chi(x,\lambda)+V_-\chi(x,-\lambda),
\end{align*}
which is a square integrable function of $x$. By the uniqueness of the $M$-function, see Proposition~\ref{prp.Mf}, this proves that the right-hand side of \eqref{eq:X9} coincides with $M_{\alpha}(s)$. The proof is complete. 
\end{proof}

\begin{proof}[Proof of Theorem~\ref{thm:a7}]
Let us consider \eqref{eq:X9} as $\lambda$ approaches $\bbR_+$ from the upper half-plane and $s$ approaches the interval $(\abs{\omega},\infty)$ from the upper half-plane. Then the Stieltjes inversion formula together with the change of variables $\dd s=\frac{\lambda}{s}\dd\lambda$ yields
\[
\begin{pmatrix}
1&\psi(s)\\
\overline{\psi(s)}&1
\end{pmatrix}
\dd\nu(s)
=
\begin{pmatrix}
s&\lambda+\ii\omega\\
\lambda-\ii\omega& s
\end{pmatrix}
\frac{\dd\sigma(\lambda)}{2s}
+
\begin{pmatrix}
s&-\lambda+\ii\omega\\
-\lambda-\ii\omega& s
\end{pmatrix}
\frac{\dd\sigma_*(\lambda)}{2s},
\]
where we have expressed the spectral measure $\dd\Sigma$ on the left by~\eqref{eq:a1}.
From here, inspecting the diagonal and off-diagonal entries, we get \eqref{eq:X3} and~\eqref{eq:X4}. 

Finally, if $\omega\not=0$, then it is easy to see from \eqref{eq:X9} that $M_\alpha(s)$ admits analytic continuation through the interval $(-\abs{\omega},\abs{\omega})$ and so the measure $\nu$ vanishes on this interval. The proof of Theorem~\ref{thm:a7} is complete. 
\end{proof}

\subsection{Proof of Theorem~\ref{thm:a6}}
As already mentioned, formula \eqref{eq.r6} is the consequence of Theorem~\ref{thm:a7} with $\omega=0$. We only need to prove the last statement. 
If $H$ is positive semi-definite, then the spectral measure $\sigma$ vanishes on $(-\infty,0)$. From \eqref{eq.r6} it follows that $\psi(s)=-1$ for $\nu$-a.e. $s<0$. Since $\psi$ is odd, we find $\psi(s)=1$ for $\nu$-a.e. $s>0$. 

Conversely, suppose $\psi(s)=1$ for $\nu$-a.e. $s>0$. From Theorem~\ref{thm:a5} we know that  $H$ is self-adjoint. From \eqref{eq.r6} it follows that the spectral measure $\sigma$ vanishes on $(-\infty,0)$, and so $H$ is positive semi-definite. 
\qed

%%%%%%%%%%%%%%%%%%%%%%%%%%%%%%%%%%%%%%%%%%%%
%%%%%%%%%%%%%%%%%%%%%%%%%%%%%%%%%%%%%%%%%%%%
\section{Example: the free case}\label{sec:examp}
%%%%%%%%%%%%%%%%%%%%%%%%%%%%%%%%%%%%%%%%%%%%
%%%%%%%%%%%%%%%%%%%%%%%%%%%%%%%%%%%%%%%%%%%%

Here we take a break from proofs and compute formulas for the spectral pair of the operator $H$ with  $q=0$. We compute explicitly the Titchmarsh--Weyl $M$-function corresponding to the hermitisation $\bH$; we will need the formula for this $M$-function in the next section. 

\subsection{$M$-functions for different boundary conditions}
We start with a general preliminary. We would like to express the $M$-function $M_\alpha$, corresponding to a general boundary parameter $\alpha$, in terms of the $M$-function $M_0$, corresponding to the Neumann boundary condition $\alpha=0$. Similarly to $M_0$, we will write the subscript $0$ to distinguish the solutions $\Phi_0$, $\Theta_0$, and $X_0$ corresponding to the Neumann boundary condition $\alpha=0$. Recall the boundary conditions \eqref{eq:phi,theta_bc_cond} for $\Phi_0$ and $\Theta_0$: 
\[
\begin{aligned}
\Phi_0(0,\lambda)&=I,\\
\Phi_0'(0,\lambda)&=0,
\end{aligned}
\qquad
\begin{aligned}
\Theta_0(0,\lambda)=0, \\
\Theta_0'(0,\lambda)=\epsilon.
\end{aligned}
\]
We need a statement similar to Lemma~\ref{lma.nonsingular}. Below $S$ and $C$ are matrices \eqref{eq:a2b} with any parameter $\alpha\in\bbC\cup\{\infty\}$.

\begin{lemma}\label{lem:inv_matr_m-func}
For all $\lambda\in\bbC\setminus\bbR$, the matrices $S-M_0(\lambda)C$ and $S-C M_0(\lambda)$ are non-singular.
\end{lemma}
\begin{proof}
Since $C$ and $S$ are self-adjoint, we have 
\[
(S-M_0(\lambda)C)^*=S-C M_0(\lambda)^*=S-C M_0(\overline{\lambda}), 
\]
and therefore it suffices to prove the non-singularity of $S-C M_0(\lambda)$ for all $\lambda\in\bbC\setminus\bbR$. Suppose to the contrary that $S-C M_0(\lambda)$ is singular for some $\lambda\in\bbC\setminus\bbR$. The solution $X_0(x,\lambda)$ satisfies 
\[
X_0(0,\lambda)=- M_0(\lambda), \quad 
X_0'(0,\lambda)=\epsilon. 
\]
Putting this together, we see that the matrix 
\[
CX_0(0,\lambda)+S\epsilon X_0'(0,\lambda)
\]
is singular. It follows that $\lambda$ is an eigenvalue of the operator $\bH$ with the boundary condition \eqref{eq:a2a}; this is impossible for a non-real $\lambda$. 
\end{proof}

\begin{lemma}
For any $\lambda\in\bbC\setminus\bbR$, we have 
\begin{equation}
M_\alpha(\lambda)=
(C+S M_0(\lambda))(S-C M_0(\lambda))^{-1}.
\label{eq:X12}
\end{equation}
Equivalently, in terms of the boundary matrix $A$, we have
\begin{equation}
\begin{aligned}
M_{\alpha}(\lambda)&=(\epsilon A+M_0(\lambda))(I-\epsilon A M_0(\lambda))^{-1}, 
&\text{ if $\alpha\not=\infty$,}
\\
M_{\infty}(\lambda)&=-\epsilon M_{0}(\lambda)^{-1}\epsilon, 
&\text{ if $\alpha=\infty$.}
\end{aligned}
\label{eq:X13}
\end{equation}
\end{lemma}

\begin{proof}
Inverting \eqref{eq:phi-theta-by-neumann}, we find
\[
\begin{aligned}
\Phi_{0}&=\Theta C + \Phi S,\\
\Theta_{0}&=\Theta S -\Phi C.
\end{aligned}
\]
Using this, we find 
\begin{align*}
X_0&=\Theta_0-\Phi_0 M_0
=(\Theta S-\Phi C)-(\Theta C+\Phi S) M_0
\\
&=\left(\Theta-\Phi (C+S M_0)(S-C M_0)^{-1}\right)(S -C M_0),
\end{align*}
where the inverse $(S-C M_0)^{-1}$ exists by Lemma~\ref{lem:inv_matr_m-func}. By the uniqueness of the $M$-function, we obtain \eqref{eq:X12}.

Let $\alpha$ be finite. Since $S$ is a constant multiple of the identity matrix, we may rewrite~\eqref{eq:X12} as
\[
 M_{\alpha}(\lambda)=(CS^{-1}+M_0(\lambda))(I-CS^{-1} M_0(\lambda))^{-1},
\]
and use that $CS^{-1}=\epsilon A$, see~\eqref{eq:a2b}, to obtain~\eqref{eq:X13}. If $\alpha=\infty$, it is evident that \eqref{eq:X13} is equivalent to \eqref{eq:X12}.
\end{proof}

\subsection{The $M$-function in the free case}
For the Neumann solutions we find
\begin{align}
\Phi_{0}(x,\lambda)&=\cos(kx)P_{+}+\cosh(kx)P_{-}, \label{eq:phi_free_neum}\\
\Theta_{0}(x,\lambda)&=\frac{\sin(kx)}{k}P_{+}-\frac{\sinh(kx)}{k}P_{-}, \label{eq:theta_free_neum}
\end{align}
where the matrices $P_{\pm}$ are as in~\eqref{eq:Ppm} and $\lambda=k^{2}$. From here it follows that 
\begin{equation}
M_{0}(\lambda)=\frac{1}{k}\left(\ii P_{+}-P_{-}\right),
\label{eq:m-func_free_neum}
\end{equation}
if $\Re k>0$ and $\Im k>0$.
Using~\eqref{eq:X13}, from here we find
\begin{equation}
M_\alpha(\lambda)=
\left(\epsilon A+\frac{1}{k}(\ii P_+ -P_-)\right)\left(I-\frac{1}{k}\epsilon A(\ii P_+ -P_-)\right)^{-1},
\label{eq:m-func_free}
\end{equation}
if $\alpha$ is finite and 
\begin{equation}
M_{\infty}(\lambda)=k(\ii P_{+}+P_{-})
\label{eq:m-func_free_dirich}
\end{equation}
if $\alpha=\infty$. 

\subsection{The spectral pair in the free case}

Using~\eqref{eq:m-func_free}, an elementary calculation yields the expression
\begin{equation}
M_\alpha(\lambda)=\frac{\epsilon A\left[k^{2}-k(P_{+}+\ii P_{-})A-\ii\right]+\ii k(P_{+}+\ii P_{-})}{(k-k_{+})(k-k_{-})},
\label{eq:m-func_free_explicit}
\end{equation}
where $\lambda=k^{2}$ with $\Re k>0$, $\Im k>0$, and
\begin{equation}
 k_{\pm}:=\frac{1+\ii}{2}\left(\Re\alpha\pm\ii\sqrt{(\Re\alpha)^{2}+2(\Im\alpha)^{2}}\right).
\label{eq:def_kpm}
\end{equation}

Having computed the $M$-function, we find $\Sigma_{\alpha}$ by using the Stieltjes inversion formula, computing separately its absolutely continuous part $\Sigma_{\alpha;\text{ac}}$ and the point part $\Sigma_{\alpha;\text{p}}$. First, from the formula for the density
\[
 \frac{\dd \Sigma_{\alpha;\text{ac}}}{\dd \lambda}(\lambda)=\frac{1}{\pi}\lim_{\varepsilon\to0+}\Im M_{\alpha}(\lambda+\ii\varepsilon),
\]
one infers for $\lambda>0$, after a straightforward calculation, that
\begin{equation}
 \frac{\dd \Sigma_{\alpha;\text{ac}}}{\dd \lambda}(\lambda)=
 \frac{\dd \nu_{\alpha;\text{ac}}}{\dd \lambda}(\lambda) 
 \begin{pmatrix}
 1 & \psi_{\alpha}(\lambda) \\ \overline{\psi_{\alpha}(\lambda)} & 1
 \end{pmatrix}
\label{eq:Sigma_ac_free}
\end{equation}
with
\begin{align}
\frac{\dd \nu_{\alpha;\text{ac}}}{\dd \lambda}(\lambda)=\frac{1+|\alpha|^{2}}{2\pi}\frac{\sqrt{\lambda}\,|\sqrt{\lambda}-\alpha|^{2}}{(\sqrt{\lambda}\Re\alpha-|\alpha|^{2})^{2}+\lambda(\sqrt{\lambda}-\Re\alpha)^{2}}, 
\quad 
\psi_{\alpha;\text{ac}}(\lambda)=\frac{\sqrt{\lambda}-\alpha}{\sqrt{\lambda}-\overline{\alpha}} \nonumber\\
\label{eq:spec_pair_free}
\end{align}
for all $\alpha\in\bbC$. In the Dirichlet case $\alpha=\infty$, we find
\begin{equation}
\frac{\dd \nu_{\infty;\text{ac}}}{\dd \lambda}(\lambda)=\frac{\sqrt{\lambda}}{2\pi}, 
\quad 
\psi_{\infty;\text{ac}}(\lambda)=1.
\label{eq:spec_pair_dirich}
\end{equation}

Second, recall that $\Sigma_{\alpha;\text{p}}$ is supported on poles of $M_{\alpha}$ continued to the real line with the weights given by the corresponding residues. One readily sees from~\eqref{eq:m-func_free_explicit} that the only possible poles of $M_{\alpha}$ may occur at $k=k_{\pm}$ provided that $\alpha\neq0$ and $\alpha\neq\infty$. Recall that $k$ in~\eqref{eq:m-func_free_explicit} is confined to the first quadrant $\Re k>0$, $\Im k>0$. By inspection of~\eqref{eq:def_kpm}, one sees that the two points $k_{\pm}$ escape the closure of the first quadrant unless $\alpha>0$. Therefore $\Sigma_{\alpha;\text{p}}=0$ for all $\alpha\in(\bbC\setminus\bbR_{+})\cup\{\infty\}$.

Suppose $\alpha\in\bbR_{+}$. Then $k_{-}=\alpha$, $k_{+}=\ii\alpha$, and~\eqref{eq:m-func_free_explicit} simplifies to 
\[
 M_{\alpha}(\lambda)=\frac{\alpha k+\ii}{k-\ii\alpha}P_{+}-\frac{\alpha k+1}{k-\alpha}P_{-}.
\]
Using these formulas, a straightforward calculation yields
\[
\Sigma_{\alpha;\text{p}}(\{\mp\alpha^2\})=\lim_{\lambda\to\mp\alpha^2}(\mp\alpha^2-\lambda)M_{\alpha}(\lambda)=2\alpha(1+\alpha^{2})P_{\pm}.
\]
Hence
\[
\Sigma_{\alpha;\text{p}}=2\alpha(1+\alpha^{2})\left(\delta_{-\alpha^2}P_{+}+\delta_{\alpha^2}P_{-}\right),
\]
where $\delta_{\lambda}$ denotes the unit mass Dirac delta measure supported on the one-point set $\{\lambda\}$.

Bearing in mind Theorem~\ref{thm:a2}, we summarize the results in terms of the spectral pair $(\nu_{\alpha}, \psi_{\alpha})$ of $H$ with $q=0$ in the next proposition. Recall that it is sufficient consider $\nu_{\alpha}$ and $\psi_{\alpha}$ on $\bbR_{+}$ due to their parities.

\begin{proposition}
Let $(\nu_{\alpha},\psi_{\alpha})$ be the spectral pair of $H$ with $q=0$.
\begin{enumerate}[i)]
\item If $\alpha\in(\bbC\setminus\bbR_{+})\cup\{\infty\}$, then $\nu_{\alpha}=\nu_{\alpha;\mathrm{ac}}$ and $\psi_{\alpha}=\psi_{\alpha;\mathrm{ac}}$, where $\nu_{\alpha;\mathrm{ac}}$ and $\psi_{\alpha;\mathrm{ac}}$ are given by formulas~\eqref{eq:spec_pair_free} and~\eqref{eq:spec_pair_dirich} for all $\lambda>0$.
\item If $\alpha\in\bbR_{+}$, then $\nu_{\alpha}=\nu_{\alpha;\mathrm{ac}}+\nu_{\alpha;\mathrm{p}}$ and $\psi_{\alpha}=\psi_{\alpha;\mathrm{ac}}+\psi_{\alpha;\mathrm{p}}$, where $\nu_{\alpha;\mathrm{ac}}$ and $\psi_{\alpha;\mathrm{ac}}$ are given by formulas~\eqref{eq:spec_pair_free} and~\eqref{eq:spec_pair_dirich} for all $0<\lambda\neq\alpha^{2}$, and 
\[
 \nu_{\alpha;\mathrm{p}}=\alpha(1+\alpha^{2})\delta_{\alpha^{2}},
 \quad
  \psi_{\alpha;\mathrm{p}}(\lambda)=\begin{cases}
&\hskip9pt0  \quad \mbox{ for }\; 0<\lambda\neq\alpha^{2},\\
&-1 \quad \mbox{ for }\; \hskip22pt \lambda=\alpha^{2}.
\end{cases}
\]
\end{enumerate}
\end{proposition}

%%%%%%%%%%%%%%%%%%%%%%%%%%%%%%%%%%%
%%%%%%%%%%%%%%%%%%%%%%%%%%%%%%%%%%%
\section{Asymptotic properties of the spectral data}
\label{sec.t}
%%%%%%%%%%%%%%%%%%%%%%%%%%%%%%%%%%%
%%%%%%%%%%%%%%%%%%%%%%%%%%%%%%%%%%%
In this section, we prove Theorems~\ref{thm:asympt_Sigma} and~\ref{thm:t2}. More precisely, we combine the asymptotic formula \eqref{eq.Masymp1} for the $M$-function with a Tauberian theorem for matrix-valued measures and deduce the asymptotic formulas for $\Sigma([0,r])$ or $\Sigma([-r,0])$, as $r\to\infty$. 

\subsection{A Tauberian theorem}

The proof of Theorem~\ref{thm:asympt_Sigma} relies on the following version of a Tauberian theorem for matrix-valued measures. This theorem is suitable for our purpose and its more general form can be found in~\cite[Thm.~2.3]{BW} without proof, referring to its scalar version~\cite[Thm.~7.5]{B_89}. For completeness, we provide another short proof inspired by the ideas used in the proof of~\cite[Thm.~6.9.1]{S3}.

\begin{theorem}\label{thm:tauber-matrix}
 Let $M$ and $M^{0}$ be Herglotz--Nevanlinna matrix-valued functions with integral representations as in~\eqref{eq:intres} in terms of matrix-valued measures $\Sigma$ and $\Sigma^{0}$, respectively. Suppose there exists a scalar-valued positive function $g=g(r)$ defined for $r>0$ such that, for any $\lambda\in\bbC\setminus\bbR$, we have 
 \[
  M(r\lambda)=g(r)\left(M^{0}(\lambda)+\smallO(1)\right)
 \]
 as $r\to\infty$. If $0$ and $1$ are not point masses of $\Sigma^{0}$, then 
 \[
  \Sigma([0,r])=rg(r)\left(\Sigma^{0}([0,1])+\smallO(1)\right)
 \]
 as $r\to\infty$. Similarly, if $0$ and $-1$ are not point masses of $\Sigma^{0}$, then 
 \[
  \Sigma([-r,0])=rg(r)\left(\Sigma^{0}([-1,0])+\smallO(1)\right)
 \]
 as $r\to\infty$.
\end{theorem}

\begin{proof}
Taking the imaginary parts, we infer from the assumptions that
\[
\lim_{r\to\infty}\frac{r}{g(r)}\int_{-\infty}^{\infty}\frac{\dd\Sigma(t)}{|t-r\lambda|^{2}}=
\int_{-\infty}^{\infty}\frac{\dd\Sigma^{0}(t)}{|t-\lambda|^{2}}
\]
for all $\lambda\in\bbC\setminus\bbR$. Changing the variable $t\mapsto rt$ in the integral on the left, the last formula can be written as
\[
\lim_{r\to\infty}\int_{-\infty}^{\infty}\frac{\dd\Sigma_{r}(t)}{|t-\lambda|^{2}}=
\int_{-\infty}^{\infty}\frac{\dd\Sigma^{0}(t)}{|t-\lambda|^{2}},
\]
where
\begin{equation}
 \Sigma_{r}(t):=\frac{\Sigma(rt)}{rg(r)}.
\label{eq:def_Sigma_r}
\end{equation}
(Here we abuse the notation slightly denoting by $\Sigma$ both the measure and its distribution function.) Putting $\lambda=1/2+\ii\sqrt{y}$, with $y>0$, we get
\[
\lim_{r\to\infty}\int_{-\infty}^{\infty}\frac{\dd\Sigma_{r}(t)}{(t-1/2)^{2}+y}=
\int_{-\infty}^{\infty}\frac{\dd\Sigma^{0}(t)}{(t-1/2)^{2}+y},
\]
or equivalently 
\begin{equation}
\lim_{r\to\infty}\int_{-\infty}^{\infty}\frac{\dd\Sigma_{r}(t+1/2)}{y+t^{2}}=
\int_{-\infty}^{\infty}\frac{\dd\Sigma^{0}(t+1/2)}{y+t^{2}}.
\label{eq:limit_rel_ala_simon_inproof}
\end{equation}

Further, we may proceed as in the proof of~\cite[Thm.~6.9.1]{S3}. By using Morera's and Vitali's theorems, we infer that the integrals in \eqref{eq:limit_rel_ala_simon_inproof} extend as analytic functions to the half-plane $\Re y>0$ and the convergence in \eqref{eq:limit_rel_ala_simon_inproof}, for $r\to\infty$, is local uniform therein.
Consequently, we may differentiate~\eqref{eq:limit_rel_ala_simon_inproof} with respect to $y$ any number of times, getting
\[
\lim_{r\to\infty}\int_{-\infty}^{\infty}\frac{\dd\Sigma_{r}(t+1/2)}{(y+t^{2})^{k}}=
\int_{-\infty}^{\infty}\frac{\dd\Sigma^{0}(t+1/2)}{(y+t^{2})^{k}}
\]
for any $k\in\bbN$. Since polynomials in the variable $1/(1+x^{2})$ are dense in the space of continuous even functions on $\bbR\cup\{\infty\}$, the last limit formula extends to 
\[
\lim_{r\to\infty}\int_{-\infty}^{\infty}h(t)\frac{\dd\Sigma_{r}(t+1/2)}{1+t^{2}}=
\int_{-\infty}^{\infty}h(t)\frac{\dd\Sigma^{0}(t+1/2)}{1+t^{2}}.
\]
for all bounded continuous even functions $h$ on $\bbR$ such that 
\[
\lim_{x\to\infty}h(x)=\lim_{x\to-\infty}h(x).
\]
Approximating the function 
\[
h(t)=(1+t^{2})\chi_{[-1/2,1/2]}(t),
\]
by continuous functions, from here one obtains 
\[
 \lim_{r\to\infty}\Sigma_{r}([0,1])=\Sigma^{0}([0,1]),
\]
if $\Sigma^{0}(\{0,1\})=0$. Recalling definition~\eqref{eq:def_Sigma_r}, the first statement follows. Proof of the second asymptotic formula for $\Sigma([-r,0])$ proceeds analogously.
\end{proof}

\subsection{Proof of Theorem~\ref{thm:asympt_Sigma}}
The proof relies on Theorem~\ref{thm:tauber-matrix} and on the asymptotics of the $M$-function from Proposition~\ref{prop:M_func_asympt}. Depending on the boundary parameter $\alpha\in\bbC\cup\{\infty\}$, we distinguish two cases.

\emph{Case $\alpha\neq\infty$:} Denote by $M_{0}^{0}$ the $M$-function of $\bH$ with $Q=0$ and $\alpha=0$.
As computed in \eqref{eq:m-func_free_neum}, we have
\begin{equation}
M_{0}^{0}(\lambda)=\frac1{k}(\ii P_{+}-P_{-}), 
\label{eq:freeMf}
\end{equation}
where $\lambda=k^2$ and $\Im k>0$, $\Re k>0$. Comparing this with the asymptotic formula  for the $M$-function from Proposition~\ref{prop:M_func_asympt}, we can write
\[
M_\alpha(\lambda)-\epsilon A
=
(1+|\alpha|^{2})M_{0}^{0}(\lambda)+\smallO(1/\sqrt{|\lambda|})
\]
as $\lambda\to\infty$ in the upper half-plane. Due to the symmetry~\eqref{eq:a4}, the same formula holds as $\lambda\to\infty$ in the lower half-plane. We can write this as
\begin{equation}
M_\alpha(r\lambda)-\epsilon A
=
\frac{1+|\alpha|^{2}}{\sqrt{r}}M_{0}^{0}(\lambda)+\smallO(1/\sqrt{r})
\label{eq:M_pi/2_asympt_inproof}
\end{equation}
as $r\to\infty$, for any $\lambda\in\bbC\setminus\bbR$ fixed. 
Now we can apply Theorem~\ref{thm:tauber-matrix} to the Herglotz--Nevanlinna functions
\[
M(\lambda):=M_{\alpha}(\lambda)-\epsilon A \quad\mbox{ and }\quad M^{0}(\lambda):=M_{0}^{0}(\lambda),
\]
with $g(r):=(1+|\alpha|^{2})/\sqrt{r}$ to conclude that 
\begin{align*}
 \Sigma([0,r])&=(1+|\alpha|^{2})\sqrt{r}\left(\Sigma^{0}([0,1])+\smallO(1)\right), \\ 
 \Sigma([-r,0])&=(1+|\alpha|^{2})\sqrt{r}\left(\Sigma^{0}([-1,0])+\smallO(1)\right)
\end{align*}
as $r\to\infty$. The spectral measure $\Sigma^{0}$ for the free Neumann case has been computed in \eqref{eq:Sigma_ac_free}, \eqref{eq:spec_pair_free}. It has no point part and its density reads 
\[
\frac{\dd\Sigma^{0}}{\dd\lambda}(\lambda)=\frac1{2\pi}\frac1{\sqrt{\abs{\lambda}}}
\begin{pmatrix}1&\sign\lambda\\ \sign\lambda&1\end{pmatrix}.
\]
Integrating over $\lambda$, we find 
\[
\Sigma^{0}([0,1])=\frac1\pi\begin{pmatrix}1&1\\1&1\end{pmatrix},
\quad
\Sigma^{0}([-1,0])=\frac1\pi\begin{pmatrix}1&-1\\-1&1\end{pmatrix}.
\]
Putting this all together, we obtain the claim of Theorem~\ref{thm:asympt_Sigma} for $\alpha\neq\infty$.

\emph{Case $\alpha=\infty$:} Although this case is to be treated separately, the arguments are analogous. 
Denoting by $M_{\infty}^{0}$ the $M$-function of $\bH$ with $Q=0$ and $\alpha=\infty$, instead of \eqref{eq:freeMf} and \eqref{eq:M_pi/2_asympt_inproof}, we find
\[
M_{\infty}^{0}(\lambda)=k(\ii P_{+}+P_{-})\quad\text{ and }\quad
M_{\infty}(r\lambda)=\sqrt{r}\left[M_{\infty}^{0}(\lambda)+\smallO(1)\right]
\]
for $r\to\infty$ and $\lambda\in\bbC\setminus\bbR$, see \eqref{eq:m-func_free_dirich}. Applying Theorem~\ref{thm:tauber-matrix} to Herglotz--Nevanlinna functions $M_{\infty}$ and $M_{\infty}^{0}$ with $g(r):=\sqrt{r}$, we obtain
\[
 \Sigma([0,r])=r^{3/2}\left(\Sigma^{0}([0,1])+\smallO(1)\right),
 \quad
 \Sigma([-r,0])=r^{3/2}\left(\Sigma^{0}([-1,0])+\smallO(1)\right)
\]
for $r\to\infty$, where this time $\Sigma^{0}$ stands for the spectral measure of the free Dirichlet case. The point part of $\Sigma^{0}$ is void and the density has been computed in~\eqref{eq:Sigma_ac_free}, \eqref{eq:spec_pair_dirich}: 
\[
\frac{\dd\Sigma^{0}}{\dd\lambda}(\lambda)=\frac1{2\pi}\sqrt{\abs{\lambda}}
\begin{pmatrix}1&\sign\lambda\\ \sign\lambda&1\end{pmatrix}.
\]
Integrating over $\lambda$, we find 
\[
\Sigma^{0}([0,1])=\frac1{3\pi}\begin{pmatrix}1&1\\1&1\end{pmatrix},
\quad
\Sigma^{0}([-1,0])=\frac1{3\pi}\begin{pmatrix}1&-1\\-1&1\end{pmatrix}.
\]
Putting this all together, we obtain the claim of Theorem~\ref{thm:asympt_Sigma} for $\alpha=\infty$. 
\qed

%%%%%%%%%%%%%%%%%%%%%%%%%%%%%%%%%%%
%%%%%%%%%%%%%%%%%%%%%%%%%%%%%%%%%%%
\section{The spectral pair at a simple eigenvalue}\label{sec.f}
%%%%%%%%%%%%%%%%%%%%%%%%%%%%%%%%%%%
%%%%%%%%%%%%%%%%%%%%%%%%%%%%%%%%%%%

\subsection{The distinguished solution}

\begin{proof}[Proof of Lemma~\ref{lma.b1}]
Since $\lambda>0$, it is also a simple eigenvalue of $\bH$. 
Consider the \emph{normalized} eigenvector $F$ of $\bH$ corresponding to this eigenvalue:
\[
\bH F=\lambda F, \qquad F=\begin{pmatrix}f_1\\ f_2\end{pmatrix}.
\]
Here $F\in L^2(\bbR_{+};\bbC^{2})$ and satisfies the boundary condition \eqref{eq:a2a0}. 
By inspection, the function 
\[
\widetilde F=\begin{pmatrix}\overline{f_2}\\ \overline{f_1}\end{pmatrix}
\]
also satisfies the differential equation $\bH \widetilde{F}=\lambda\widetilde{F}$ and the boundary condition \eqref{eq:a2a0}. Since, by assumption, $\lambda$ is a simple eigenvalue of $\bH$, we have
\[
\begin{pmatrix} \overline{f_2}\\ \overline{f_1}\end{pmatrix}
=
e^{\ii\vartheta}
\begin{pmatrix} f_1\\f_2\end{pmatrix}
\]
with some unimodular complex number $e^{\ii\vartheta}$. Then we see that
\[
e:=\sqrt{2}e^{-\ii\vartheta/2}\,\overline{f_1}
\]
satisfies the ``anti-linear eigenvalue equation'' \eqref{f1} and the boundary condition \eqref{eq.r4b}. The factor $\sqrt{2}$ ensures that $e$ has norm one. At this point we also note that the normalised eigenvector of $\bH$ can be written as
\[
\frac1{\sqrt{2}}
\begin{pmatrix}\overline{e}\\ e\end{pmatrix}.
\]

Let us prove the uniqueness of $e$, up to multiplication by $\pm1$. Suppose that two normalised solutions $e_1$ and $e_2$ satisfy the required assumptions. Then $\bH$ has normalised eigenvectors
\[
\frac1{\sqrt{2}}
\begin{pmatrix}\overline{e_1} \\e_1\end{pmatrix}
\quad \text{ and }\quad
\frac1{\sqrt{2}}
\begin{pmatrix}\overline{e_2} \\e_2\end{pmatrix}
\]
both corresponding to the eigenvalue $\lambda$. By the assumption that $\lambda$ is simple, we have
\[
\begin{pmatrix}\overline{e_1} \\e_1\end{pmatrix}
=e^{\ii\vartheta}
\begin{pmatrix}\overline{e_2} \\e_2\end{pmatrix}
\]
with some unimodular complex number $e^{\ii\vartheta}$. Since
\[
e_1=e^{\ii\vartheta} e_2\quad\text{ and }\quad
\overline{e_1}=e^{\ii\vartheta}\,\overline{e_2}
\]
for non-zero $e_1$ and $e_2$, we conclude that $e^{\ii\vartheta}=e^{-\ii\vartheta}$, i.e. $e^{\ii\vartheta}=\pm1$.
\end{proof}

\subsection{Computing $\Sigma(\{\lambda\})$ for a simple singular value $\lambda>0$}
Here we compute the spectral measure $\Sigma(\{\lambda\})$ in the case when $\lambda>0$ is a simple eigenvalue of $\abs{H}$. 

We start by recalling the explicit formula for the resolvent kernel of $\bH$. 
For $\lambda\in\bbC\setminus\bbR$, the resolvent $(\bH-\lambda)^{-1}$ is an integral operator acting on $L^2(\bbR_+;\bbC^2)$; let $\bR(x,y;\lambda)$ be the $2\times 2$ matrix-valued integral kernel of this operator. We recall that $\Phi$ and $X$ are the solutions of the eigenvalue equation \eqref{eq:init_diff_eq} specified by \eqref{eq:phi,theta_bc_cond} and \eqref{eq:X}. 

\begin{proposition}\label{prop:resolvent_kernel}
For all $\lambda\in\bbC\setminus\bbR$, we have 
\begin{equation}
\bR(x,y;\lambda)
=
\begin{cases}
-X(x,\lambda)\Phi(y,\overline{\lambda})^*, \quad x\geq y,
\\
-\Phi(x,\lambda)X(y,\overline{\lambda})^*, \quad x<y.
\end{cases}
\label{e13}
\end{equation}
\end{proposition}
Formulas of this genre are well known, but for completeness we outline the proof in Appendix, see Section~\ref{sec.reskernel}.

\begin{lemma}\label{lma.e2}
Let $\lambda>0$ be a simple eigenvalue of $\bH$ with a normalised eigenvector $F_{\lambda}\in L^2(\bbR_+;\bbC^2)$. Then the point mass of the spectral measure $\Sigma$ at $\lambda$ is the $2\times 2$ matrix
\begin{equation}
\Sigma(\{\lambda\})=\jap{\cdot,L_\alpha(F_{\lambda})}_{\bbC^2}\, L_\alpha(F_{\lambda}),
\label{e12}
\end{equation}
where
\[
L_\alpha(F_{\lambda}):=C\epsilon F_{\lambda}'(0)-S F_{\lambda}(0).
\]
\end{lemma}

\begin{proof}
From the integral representation \eqref{eq:intres} for $M_\alpha(z)$ it follows that  
\begin{equation}
-\ii\delta M_\alpha(\lambda+\ii\delta)\to \Sigma(\{\lambda\})
\label{e20}
\end{equation}
as $\delta\to0_+$. 
Next, by Proposition~\ref{prop:resolvent_kernel} we can write the resolvent kernel as
\[
\bR(x,y;z)=\bR_1(x,y;z)+\bR_2(x,y;z),
\]
where
\begin{equation}
\bR_1(x,y;z):=\Phi(x,z)M_\alpha(z)\Phi(y,\overline{z})^*
\label{eq:def_R1}
\end{equation}
and
\begin{equation}
\bR_2(x,y;z):=
\begin{cases}
-\Theta(x,z)\Phi(y,\overline{z})^*, \quad x\geq y,
\\
-\Phi(x,z)\Theta(y,\overline{z})^*, \quad x<y.
\end{cases}
\label{eq:def_R2}
\end{equation}
Since $\bR_2(x,y;\cdot)$ is an entire function for all $x,y$, we deduce from here and \eqref{e20} that
\begin{equation}
-\ii\delta \bR(x,y;\lambda+\ii\delta)\to \Phi(x,\lambda)\Sigma(\{\lambda\})\Phi(y,\lambda)^*
\label{e21}
\end{equation}
as $\delta\to0_+$, uniformly over $x,y$ in compact subsets of $\bbR_+$.

On the other hand, by the spectral theorem for self-adjoint operators, 
\[
-\ii\delta (\bH-\lambda-\ii\delta)^{-1}\to \jap{\cdot, F_{\lambda}}_{L^2}F_{\lambda},
\]
in the weak operator topology as $\delta\to0_+$. Here the right-hand side is the rank one projection in $L^2(\bbR_+;\bbC^2)$ onto the one-dimensional eigenspace spanned by the eigenvector $F_{\lambda}$. Comparing this with \eqref{e21}, we conclude that 
\begin{equation}
\Phi(x,\lambda)\Sigma(\{\lambda\})\Phi(y,\lambda)^*
=
F_{\lambda}(x)F_{\lambda}(y)^*.
\label{e22}
\end{equation}
Now let $L_\alpha$ be the differential expression \eqref{e11a}. 
By the boundary condition \eqref{eq:phi,theta_bc_cond}, we have 
\[
L_\alpha\left(\Phi(\cdot,\lambda)\right)=-I.
\]
Using this and applying $L_\alpha$ to \eqref{e22} in both $x$ and $y$ variables, we arrive at the desired conclusion \eqref{e12}. 
The proof of Lemma~\ref{lma.e2} is complete. 
\end{proof}

\subsection{Proof of Theorem~\ref{thm.5.4}}
Suppose, to get a contradiction, that $\ell_\alpha(e)=0$. Then from here and the boundary condition \eqref{eq.r4b} we find $e(0)=e'(0)=0$. Considering the anti-linear eigenvalue equation \eqref{f1} for $e$ as a system of two equations for the real and imaginary parts of $e$, we conclude that $e=0$, which contradicts the normalisation of $e$. Thus, $\ell_\alpha(e)\not=0$.

As discussed in the proof of Lemma~\ref{lma.b1}, we can write the normalised eigenvector of $\bH$ corresponding to the eigenvalue $\lambda$ as 
\[
F_\lambda=\frac1{\sqrt{2}}\begin{pmatrix}\overline{e}\\ e\end{pmatrix}.
\]
By Lemma~\ref{lma.e2} and definitions~\eqref{eq:a2b} and~\eqref{e11a}, we have 
\[
L_\alpha(F_{\lambda})=
\frac1{\sqrt{2}}\begin{pmatrix}\overline{\ell_\alpha(e)}\\ {\ell_\alpha(e)}\end{pmatrix}
\]
and
\[
\Sigma(\{\lambda\})
=
\frac12
\begin{pmatrix}
\abs{\ell_\alpha(e)}^2& \overline{\ell_\alpha(e)^2}
\\
\ell_\alpha(e)^2&\abs{\ell_\alpha(e)}^2
\end{pmatrix}.
\]
From here and the representation \eqref{eq:a1} for the measure $\Sigma$ in terms of $\nu$ and $\psi$, we get the required formulas \eqref{e11}.
\qed

%%%%%%%%%%%%%%%%%%%%%%%%%%%%%%%%%
%%%%%%%%%%%%%%%%%%%%%%%%%%%%%%%%%
\section{Concluding remarks: possible extensions}\label{sec:misc}
%%%%%%%%%%%%%%%%%%%%%%%%%%%%%%%%%
%%%%%%%%%%%%%%%%%%%%%%%%%%%%%%%%%

\subsection{Schr{\" o}dinger operators on compact intervals}
Our results are stated for Schr{\" o}dinger operators $H$ on the half-line $\bbR_{+}$. However, little additional effort is needed to adapt these results to operators $H$ on the interval $(0,b)$ with $b<\infty$ and an additional boundary condition at the right endpoint $b$. Below we briefly discuss the minor modifications necessary to adapt to this case.

For $b$ finite, the operator $H$ is defined by~\eqref{eq:def_H} on functions from $W^{2,2}([0,b])$ that satisfy the boundary conditions 
\begin{equation}
\begin{aligned}
f'(0)+\alpha f'(0)&=0,\\
f'(b)+\beta f'(b)&=0,
\end{aligned}
\label{eq:bc_interval}
\end{equation}
where $\alpha,\beta\in\bbC\cup\{\infty\}$ are boundary parameters. The corresponding operator $\bH$ is still given by~\eqref{a0} but is acting on functions $F\in W^{2,2}([0,b];\bbC^{2})$ satisfying the boundary conditions~\eqref{eq:a2a0} and~\eqref{eq:bc_b}. The $M$-function $M_{\alpha}(b;\lambda)$ is now defined by~\eqref{eq:def_M_b} and the solution $X$ is replaced by $X_{b}$ defined in~\eqref{eq:def_chi_b}; the dependence on $\beta$ is suppressed by the notation. Instead of the square integrability at infinity, $X_{b}$ satisfies the boundary condition~\eqref{eq:bc_b} at the right endpoint $b$. The $M$-function $M_{\alpha}(b;\lambda)$ has the same integral representation as in~\eqref{eq:intres}, giving rise to the spectral measure~$\Sigma$. Moreover, the asymptotic formulas of Proposition~\ref{prop:M_func_asympt} remain valid in the case $b<\infty$ independently of the boundary condition at $b$, see Remark~\ref{rem:M_asympt_compact}.

Theorem~\ref{thm:a2} holds in exactly the same form since its proof relies on the symmetries of $\bH$ which remain valid for finite $b$. As a result, we have the notion of spectral pair $(\nu,\psi)$ assigned to $\bH$. Let us remark that the spectral measure $\Sigma$ is always discrete (see for example \cite[Chp.~9]{Atk}) and is supported on isolated eigenvalues of $\bH$ that form a sequence accumulating at $\pm\infty$. It follows that the even measure $\nu$ is supported on points $\pm\lambda_{n}$, where $0\leq\lambda_{0}<\lambda_{1}<\dots$ are the singular values of $H$, i.e. the eigenvalues of $|H|$. Their multiplicities can be either 1 or 2. If a singular value $\lambda>0$ is simple, the formulas from~\eqref{e11} still hold with the normalised solution $e\in L^{2}(0,b)$ satisfying equation~\eqref{f1} and boundary conditions~\eqref{eq:bc_interval}. Such distinguished solution $e$ exists and is unique up to a sign which can be proved similarly to Lemma~\ref{lma.b1}. 

The remaining statements remain true for the case of finite $b$, and the proofs only require minimal notational changes: $M_\alpha(\lambda)\mapsto M_\alpha(b;\lambda)$ and $X\mapsto X_b$.

\subsection{Other classes of potentials}

Throughout this work, we have assumed that the complex potential $q$ is bounded. This assumption was made primarily in order to keep the size of the paper within reasonable bounds. We expect that most of the construction can be extended to suitable classes of unbounded potentials. In the discrete case this has been achieved in \cite{ELY} for the Jacobi operators in the limit-point class. 

It is an interesting question to specify our analysis to the case of potentials $q$ with a suitable decay at infinity and to relate the spectral pair $(\nu,\psi)$ to the standard objects of scattering theory. We plan to address this question in future work.

\appendix

\section*{Appendix}

For completeness we provide proofs of statements from the Titchmarsh--Weyl theory for the matrix-valued operator $\bH$; these are Propositions~\ref{prp.Mf}, \ref{prop:M_func_asympt}, \ref{thm.sp.decomp} and \ref{prop:resolvent_kernel}. All these statements can be regarded as known. They can be reduced to the theory of Hamiltonian systems (see e.g. \cite{ClarkGHL} for such reduction) developed in great generality by Hinton and Shaw \cite{HS_81}. However, we found that extracting relevant facts from this literature and translating them into the language of spectral theory of operator $\bH$ is not an easy task, and therefore here we indicate the main steps of the proofs. We focus on those steps that are different from the scalar case (i.e. the spectral theory of the scalar self-adjoint Schr\"odinger operator $H$). For the scalar case, the standard reference is \cite{CodLev}.

\setcounter{equation}{0}
\setcounter{section}{1}

\subsection{Identities for $\Phi$ and $\Theta$}
We start by recording two identities for the solutions $\Phi$ and $\Theta$ that will be used throughout the appendix. 

\begin{lemma}
For any $\lambda\in\bbC$ and any $x\geq0$, we have
\begin{align}
\Theta(x,\lambda)\Phi(x,\overline{\lambda})^*&=\Phi(x,\lambda)\Theta(x,\overline{\lambda})^*,
\label{eq.X3}
\\
\Theta'(x,\lambda)\Phi(x,\overline{\lambda})^*&=\Phi'(x,\lambda)\Theta(x,\overline{\lambda})^*+\epsilon.
\label{eq.X4}
\end{align}
\end{lemma}
\begin{proof}
Fix $y\geq0$ and denote 
\[
R(x):=\Theta(x,\lambda)\Phi(y,\overline{\lambda})^*-\Phi(x,\lambda)\Theta(y,\overline{\lambda})^*.
\]
We need to prove the relations
\begin{equation}
R(y)=0, \quad R'(y)=\epsilon.
\label{e15}
\end{equation}
Clearly, $R$ satisfies the eigenvalue equation~\eqref{eq:init_diff_eq}. 
Using the Wronskian identities~\eqref{eq:wronsk_1} and~\eqref{eq:wronsk_2}, we check that 
\begin{align}
[\Phi(\cdot,\overline{\lambda})^*,R]=\Phi(y,\overline{\lambda})^*,\quad
[\Theta(\cdot,\overline{\lambda})^*,R]=\Theta(y,\overline{\lambda})^*.
\label{e18}
\end{align}
Now let $R_*$ be the solution to the eigenvalue equation~\eqref{eq:init_diff_eq} satisfying the initial conditions \eqref{e15} at $x=y$. We need to check that $R=R_*$. 

Computing the Wronskian at the point $x=y$, we observe that 
\[
[\Phi(\cdot,\overline{\lambda})^*,R_*]=\Phi(y,\overline{\lambda})^*, \quad 
[\Theta(\cdot,\overline{\lambda})^*,R_*]=\Theta(y,\overline{\lambda})^*.
\]
Subtracting from \eqref{e18}, we find
\[
[\Phi(\cdot,\overline{\lambda})^*,R-R_*]=0, \quad [\Theta(\cdot,\overline{\lambda})^*,R-R_*]=0.
\]
From here it is easy to conclude that $R-R_*=0$. 
\end{proof}

\subsection{Existence of $M_\alpha(\lambda)$: proof of Proposition~\ref{prp.Mf}}\label{sec:app1}
As in the scalar case, we start by considering the eigenvalue equation 
\begin{equation}
-\epsilon F''+QF=\lambda F
\label{eq:diff_F}
\end{equation}
on the compact interval $[0,b]$, $b>0$, endowed with the boundary condition 
\eqref{eq:a2a0} at zero and the additional boundary condition 
\[
F(b)=0,
\]
at $x=b$. (One could consider a more general boundary condition at $x=b$, as in \eqref{eq:bc_interval}, but in the limit $b\to\infty$, the dependence on this boundary condition disappears.) 
As in \eqref{eq:def_M_b} with $\beta=\infty$, for $\lambda\in\bbC\setminus\bbR$ we define
\begin{equation}
 M_\alpha(b;\lambda):=\Phi(b,\lambda)^{-1}\Theta(b,\lambda);
\label{eq:def_M_b-Dir}
\end{equation}
the inverse exists by Lemma~\ref{lma.nonsingular}.
The function $M_\alpha(b;\cdot)$ is meromorphic in $\lambda$ with poles on the real axis. 
We will show that $M_\alpha(b;\lambda)$ converges to a limit as $b\to\infty$ for any $\lambda\in\bbC\setminus\bbR$; the $M$-function $M_\alpha(\lambda)$ will be identified with this limit. 

We begin by establishing several auxiliary facts. Let $X_{b}(x,\lambda)$ be the solution defined by \eqref{eq:def_chi_b}. 

\begin{lemma}\label{lem:Im_M_b_id}
 For all $b>0$ and $\lambda\in\bbC\setminus\bbR$, we have
 \begin{equation}
  \frac{\Im M_\alpha(b;\lambda)}{\Im\lambda}=\int_{0}^{b}X_{b}(x,\lambda)^{*}X_{b}(x,\lambda)\dd x.
 \label{eq:Im_M_b_id}
 \end{equation}
\end{lemma}

\begin{proof}
By differentiating the Wronskian $[X_{b}(\cdot,\lambda)^{*},X_{b}(\cdot,\lambda)]$ with respect to $x$ and using that $X_{b}(x,\lambda)$ solves~\eqref{eq:diff_F}, we find that 
\[
 [X_{b}(\cdot,\lambda)^{*},X_{b}(\cdot,\lambda)]'=(\overline{\lambda}-\lambda)X_{b}(x,\lambda)^{*}X_{b}(x,\lambda).
\]
By integrating and taking into account that $[X_{b}(b,\lambda)^{*},X_{b}(b,\lambda)]=0$, we obtain
\[
(\overline{\lambda}-\lambda)\int_{0}^{b}X_{b}(x,\lambda)^{*}X_{b}(x,\lambda)\dd x=-[X_{b}(0,\lambda)^{*},X_{b}(0,\lambda)].
\]
A short calculation using the boundary conditions~\eqref{eq:phi,theta_bc_cond} yields
\[
[X_{b}(0,\lambda)^{*},X_{b}(0,\lambda)]=M_\alpha(b;\lambda)-M_\alpha(b;\lambda)^{*},
\]
which completes the proof.
\end{proof}

Using definition~\eqref{eq:def_chi_b}, we can rewrite the right-hand side of~\eqref{eq:Im_M_b_id} as 
\[
\begin{pmatrix}
-M_\alpha(b,\lambda)^{*} & I
\end{pmatrix}
\int_{0}^{b}
\begin{pmatrix}
\Phi(x,\lambda)^{*}\Phi(x,\lambda) & \Phi(x,\lambda)^{*}\Theta(x,\lambda) \\
\Theta(x,\lambda)^{*}\Phi(x,\lambda) & \Theta(x,\lambda)^{*}\Theta(x,\lambda)
\end{pmatrix}\dd x
\begin{pmatrix}
-M_\alpha(b,\lambda) \\ I
\end{pmatrix}.
\]
In particular, the matrix $M=M_\alpha(b;\lambda)$ satisfies the inequality 
\begin{equation}
\begin{pmatrix}
-M^{*} & I
\end{pmatrix}
\int_{0}^{b}
\begin{pmatrix}
\Phi(x,\lambda)^{*}\Phi(x,\lambda) & \Phi(x,\lambda)^{*}\Theta(x,\lambda) \\
\Theta(x,\lambda)^{*}\Phi(x,\lambda) & \Theta(x,\lambda)^{*}\Theta(x,\lambda)
\end{pmatrix}\dd x
\begin{pmatrix}
-M \\ I
\end{pmatrix}\leq\frac{\Im M}{\Im\lambda}.
\label{eq:quadr_ineq_M}
\end{equation}
This is used as a motivation for the definition of the sets
\[
\mathcal{D}(b;\lambda):=\{M\in\bbC^{2,2} \mid M \mbox{ satisfies~\eqref{eq:quadr_ineq_M}}\}
\]
for $b>0$ and $\lambda\in\bbC\setminus\bbR$.
The inequality $\leq$ in~\eqref{eq:quadr_ineq_M} is chosen so that the sets $\mathcal{D}(b;\lambda)$ turn out to be bounded in the end. The sets $\mathcal{D}(b;\lambda)$ are matrix analogues to the Weyl disks from the limit-point/limit-circle analysis of the scalar Sturm--Liouville equations. Of course, $M_\alpha(b;\lambda)\in\mathcal{D}(b;\lambda)$.

\begin{lemma}
For all $b>0$ and $\lambda\in\bbC\setminus\bbR$, sets $\mathcal{D}(b;\lambda)$ are compact and nesting, i.e. $\mathcal{D}(b_2;\lambda)\subset\mathcal{D}(b_1;\lambda)$, if $0<b_1<b_2$.
\end{lemma}
\begin{proof}
It is clear that $\mathcal{D}(b;\lambda)$ is closed; let us prove that it is bounded. Denote 
\begin{equation}
\int_{0}^{b}
\begin{pmatrix}
\Phi(x,\lambda)^{*}\Phi(x,\lambda) & \Phi(x,\lambda)^{*}\Theta(x,\lambda) \\
\Theta(x,\lambda)^{*}\Phi(x,\lambda) & \Theta(x,\lambda)^{*}\Theta(x,\lambda)
\end{pmatrix}\dd x
=
\begin{pmatrix}
A_{11}&A_{12} \\
A_{21}&A_{22}
\end{pmatrix}\, ;
\label{eq:*1}
\end{equation}
then \eqref{eq:quadr_ineq_M} can be written as 
\begin{equation}
M^*A_{11}M\leq A_{21}M+M^*A_{12}-A_{22}+\Im M/\Im\lambda. 
\label{eq:*2}
\end{equation}
Notice that $A_{11}$ is positive definite; the strict definiteness is a~consequence of invertibility of matrix $\Phi(x,\lambda)$. Now suppose, to get a contradiction, that there exists a sequence $\{M_n\}$ of matrices satisfying \eqref{eq:*2} with $M=M_n$ and $\norm{M_n}\to\infty$. Then the norm of the left-hand side of \eqref{eq:*2} is at least $c\norm{M_n}^2$, where $c$ is the smallest eigenvalue of $A_{11}$, while the norm of the right-hand side grows linearly in $\norm{M_n}$; this is a contradiction. Thus,  $\mathcal{D}(b;\lambda)$ is bounded. 

Since the matrix in the integral in \eqref{eq:*1} is positive semi-definite, the left-hand side of~\eqref{eq:quadr_ineq_M} is non-decreasing in $b$, which implies the nesting property of $\mathcal{D}(b;\lambda)$.
\end{proof}

\begin{proof}[Proof of Proposition~\ref{prp.Mf}]
Fix a non-real $\lambda$; let us first prove the existence of a~matrix $M_\alpha(\lambda)$ with the property $X(\cdot,\lambda)\in L^{2}(\bbR_{+};\bbC^{2,2})$, see \eqref{eq:X}. Take any sequence $b_{n}\to\infty$ and consider the corresponding sequence of matrices $\{M_\alpha(b_{n};\lambda)\}$. By the nesting property of the previous lemma, all these matrices belong to the bounded set $\calD(b_1;\lambda)$. Thus, by a compactness argument, passing to a subsequence, we may assume that our sequence of matrices converges. Let us denote the limiting matrix by $M_\alpha(\lambda)$ and let the solution $X(x,\lambda)$ be as in \eqref{eq:X}. For any $b>0$ and any $b_n>b$ we have 
\[
 \int_{0}^{b}X_{b_n}(x,\lambda)^{*}X_{b_n}(x,\lambda)\dd x\leq  \int_{0}^{b_{n}}X_{b_n}(x,\lambda)^{*}X_{b_n}(x,\lambda)\dd x=\frac{\Im M_\alpha(b_n;\lambda)}{\Im\lambda}
\]
by Lemma~\ref{lem:Im_M_b_id}. Sending $n\to\infty$, we find that 
\begin{equation}
 \int_{0}^{b}X(x,\lambda)^{*}X(x,\lambda)\dd x\leq\frac{\Im M_\alpha(\lambda)}{\Im\lambda}.
\label{eq:X2a}
\end{equation}
Since this holds true for any $b>0$, we see that $X(\cdot,\lambda)\in L^2(\bbR_{+};\bbC^{2,2})$, as required. 

As already mentioned in Section~\ref{sec.b2}, the uniqueness of $M_\alpha(\lambda)$ follows easily from the dimension lemma (Lemma~\ref{lem:dim_S}). 

From the above construction it is clear that $M_\alpha(\lambda)$ is the (locally uniform) limit of the sequence of meromorphic functions $M_\alpha(b_n;\lambda)$. Thus, $M_\alpha(\lambda)$ is holomorphic in the upper and lower half-planes. It follows from~\eqref{eq:X2a} that for $\Im\lambda>0$ we have $\Im M_\alpha(\lambda)\geq0$. The latter inequality is actually strict because the columns of $X(x,\lambda)$ are linearly independent.

It remains to check the identity $M_\alpha(\overline{\lambda})=M_\alpha(\lambda)^*$. 
By the definition \eqref{eq:def_M_b-Dir} of $M_\alpha(b;\overline{\lambda})$, relation \eqref{eq.X3} can be rewritten as
\[
M_\alpha(b;\overline{\lambda})=M_\alpha(b;\lambda)^*.
\]
Passing to the limit $b\to\infty$, we obtain $M_\alpha(\overline{\lambda})=M_\alpha(\lambda)^*$. 
\end{proof}

\begin{remark}\label{rem:lim_point}
The first part of the proof shows that for any $\lambda\in\bbC\setminus\bbR$, the set $\mathcal{D}(b;\lambda)$ shrinks to a point as $b\to\infty$ (justifying the terminology \emph{limit-point case}).
\end{remark}

\subsection{Asymptotic formula for $M_\alpha(\lambda)$: proof of Proposition~\ref{prop:M_func_asympt}}
\label{sec.masympt}
We start with a simple application of the resolvent identity. Below $M_0$ is the $M$-function corresponding to $\alpha=0$ and $M_0^{(0)}$ is the $M$-function corresponding to $\alpha=0$ and $Q=0$. We denote 
\[
\norm{Q}_{L^\infty}=\esssup_{x>0}\norm{Q(x)},
\]
where $\norm{Q(x)}$ is the matrix norm of the $2\times2$ matrix $Q(x)$. 

\begin{lemma}
For any $\lambda\in\bbC\setminus\bbR$, the estimate
\begin{equation}
\norm{M_0(\lambda)-M_0^{(0)}(\lambda)}
\leq
\frac{\norm{Q}_{L^\infty}}{\abs{\Im\lambda}}\Big|\Im\Tr M_0(\lambda)\Big|^{1/2}\left|\Im\Tr M_0^{(0)}(\lambda)\right|^{1/2}
\label{eq:x1}
\end{equation}
holds true. 
\end{lemma}
\begin{proof}
For the rest of the proof, we write $M(\lambda)$ in place of $M_0(\lambda)$ and $M^{(0)}(\lambda)$ in place of $M_0^{(0)}(\lambda)$ for readability. We need some notation. Let $\bH$ be the operator \eqref{a0} with $\alpha=0$, and let $\bH^{(0)}$ be the operator corresponding to $\alpha=0$ and $Q=0$. We denote by $\bR$ the resolvent kernel of $\bH$ (see \eqref{e13}) and by $\bR^{(0)}$ the resolvent kernel of $\bH^{(0)}$. The proof uses two standard resolvent identities, 
\begin{align}
(\bH-\lambda)^{-1}-(\bH^{(0)}-\lambda)^{-1}&=-(\bH^{(0)}-\lambda)^{-1}Q(\bH-\lambda)^{-1},
\label{eq:x2}
\\
\Im(\bH-\lambda)^{-1}&=(\Im \lambda)(\bH-\overline{\lambda})^{-1}(\bH-\lambda)^{-1}.
\label{eq:x3}
\end{align}
Writing \eqref{eq:x2} in terms of resolvent kernels and using that $M(\lambda)=\bR(0,0;\lambda)$, we find 
\[
M(\lambda)-M^{(0)}(\lambda)=-\int_{-\infty}^\infty \bR^{(0)}(0,x;\lambda)Q(x)\bR(x,0;\lambda)\dd x\, .
\]
By Cauchy--Schwarz this yields
\[
\norm{M(\lambda)-M^{(0)}(\lambda)}
\leq
\norm{Q}_{L^\infty}
\biggl(\int_{-\infty}^\infty \norm{\bR^{(0)}(0,x;\lambda)}^2\dd x\biggr)^{1/2}
\biggl(\int_{-\infty}^\infty \norm{\bR(x,0;\lambda)}^2\dd x\biggr)^{1/2}. 
\]
For any $2\times2$ matrix $X$, we have $\norm{X}^2\leq\Tr X^*X$. Using this and the resolvent identity \eqref{eq:x3}, for the second integral in the last estimate we find
\begin{align*}
\int_{-\infty}^\infty &\norm{\bR(x,0;\lambda)}^2\dd x
\leq
\int_{-\infty}^\infty \Tr\bigl(\bR(x,0;\lambda)^*\bR(x,0;\lambda)\bigr)\dd x
\\
&=\int_{-\infty}^\infty \Tr\bigl(\bR(0,x;\overline{\lambda})\bR(x,0;\lambda)\bigr)\dd x
=\frac1{\Im\lambda}\Im\Tr\bR(0,0;\lambda)
\\
&=\frac1{\Im\lambda}\Im\Tr M(\lambda), 
\end{align*}
and in the same way 
\[
\int_{-\infty}^\infty\norm{\bR^{(0)}(0,x;\lambda)}^2\dd x
\leq 
\frac1{\Im\lambda}\Im\Tr M^{(0)}(\lambda).
\]
Putting this together, we obtain \eqref{eq:x1}. 
\end{proof}

\begin{proof}[Proof of Proposition \ref{prop:M_func_asympt}]
First consider the case $\alpha=0$. Recall that for the $M$-function in the free case $Q=0$ we have the exact expression \eqref{eq:m-func_free_neum}. Putting this together with \eqref{eq:x1}, by a simple bootstrapping argument we obtain 
\[
M_0(\lambda)=\frac{1}{k}\left(\ii P_{+}-P_{-}\right)+\mathcal{O}(k^{-3}),
\]
as required. Next, let $\alpha$ be finite and non-zero. From~\eqref{eq:X13}, using the matrix identity $(\epsilon A)^{2}=|\alpha|^{2}I$,  we easily deduce the expansion 
\[
 M_{\alpha}(\lambda)
 =\epsilon A+(1+\abs{\alpha}^2)M_{0}(\lambda)+(1+\abs{\alpha}^2)M_0(\lambda)\epsilon A M_0(\lambda) +\mathcal{O}(k^{-3}). 
\]
Substituting the above expansion for $M_0(\lambda)$, we arrive at~\eqref{eq.Masymp1}. Finally, from 
\[
M_{\infty}(\lambda)=-\epsilon M_{0}(\lambda)^{-1}\epsilon
\]
we obtain the asymptotic formula~\eqref{eq.Masymp2} for the case $\alpha=\infty$.
\end{proof}

\subsection{The resolvent kernel of $\bH$: proof of Proposition~\ref{prop:resolvent_kernel}}
\label{sec.reskernel}
We first prepare two identities. Using the definition of $X(x,\lambda)$ and the identity $M_\alpha(\overline{\lambda})=M_\alpha(\lambda)^*$, from \eqref{eq.X3}, \eqref{eq.X4} we obtain 
\begin{align}
X(x,\lambda)\Phi(x,\overline{\lambda})^*&=\Phi(x,\lambda)X(x,\overline{\lambda})^*,
\label{eq.X5}
\\
X'(x,\lambda)\Phi(x,\overline{\lambda})^*&=\Phi'(x,\lambda)X(x,\overline{\lambda})^*+\epsilon.
\label{eq.X6}
\end{align}
Next, let $\bR$ denote the right-hand side of \eqref{e13}; we need to check that $\bR$ coincides with the resolvent kernel of $\bH$. Let $F\in L^2(\bbR_+;\bbC^2)$ be a~function compactly supported in $\bbR_+$ (in particular, $F$ is vanishing near the origin), and define
\begin{align*}
G(x)&:=\int_0^\infty \bR(x,y;\lambda)F(y)\dd y
\\
&=
-X(x,\lambda)\int_0^x \Phi(y,\overline{\lambda})^*F(y)\dd y
-\Phi(x,\lambda)\int_x^\infty X(y,\overline{\lambda})^*F(y)\dd y
\end{align*}
for $x\geq0$.
Using \eqref{eq.X5} and \eqref{eq.X6},  we compute for $x>0$:
\begin{align*}
G'(x)&=
-X'(x,\lambda)\int_0^x \Phi(y,\overline{\lambda})^*F(y)\dd y
-\Phi'(x,\lambda)\int_x^\infty X(y,\overline{\lambda})^*F(y)\dd y,
\\
G''(x)&=
-X''(x,\lambda)\int_0^x \Phi(y,\overline{\lambda})^*F(y)\dd y
-\Phi''(x,\lambda)\int_x^\infty X(y,\overline{\lambda})^*F(y)\dd y
-\epsilon F(x).
\end{align*}
Using the differential equation for $X$ and $\Phi$, we find that $G$ satisfies
\[
-\epsilon G''(x)+Q(x)G(x)-\lambda G(x)=F(x).
\]
It is easy to see directly that $G\in L^2(\bbR_+;\bbC^2)$ and that $G$ satisfies the boundary conditions for $\bH$ at $x=0$. It follows that $G\in \Dom \bH$ and 
\[
(\bH-\lambda)G=F. 
\]
This shows that the integral operator with the kernel $\bR(x,y;\lambda)$ is indeed the resolvent of $\bH$. 
The proof of Proposition~\ref{prop:resolvent_kernel} is complete. \qed

\subsection{The diagonalisation of $\bH$: proof of Proposition~\ref{thm.sp.decomp}}
\label{sec:app3}

We prove Proposition~\ref{thm.sp.decomp} in three steps. Below $C_0^\infty(\bbR_+;\bbC^2)$ is the space of smooth compactly supported functions on $\bbR_+$ with values in $\bbC^2$. 

\emph{Step 1: $U$ is an isometry.}
As in the proof of of Lemma~\ref{lma.e2}, we write the resolvent kernel of $\bH$ for $z\in\bbC\setminus\bbR$ in the form
\[
\bR(x,y;z)=\bR_1(x,y;z)+\bR_2(x,y;z),
\]
where $\bR_1$ and $\bR_2$ are defined by~\eqref{eq:def_R1} and \eqref{eq:def_R2}, respectively. Fix $F\in C_0^\infty(\bbR_+;\bbC^2)$. Let $(a,b)\subset \bbR$ be a bounded interval such that $\Sigma(\{a\})=\Sigma(\{b\})=0$. Consider the limit
\[
I(F):=\frac1\pi\lim_{\eps\to0+} \Im \int_a^b \jap{(\bH-\lambda-\ii\eps)^{-1}F,F}\dd\lambda. 
\]
It is easy to see that in the limit $\eps\to0_+$ the kernel $\bR_2$ is Hermitian, and therefore its contribution to $I(F)$ vanishes.  Consider the contribution of $\bR_1$. 
Let 
\[
(UF)(z):=\int_0^\infty \Phi(x,\overline{z})^{*}F(x)\dd x
\]
for $z=\lambda+\ii\eps$ with $\lambda\in\bbR$.
Observe that
\[
\lim_{\eps\to0+}(UF)(\lambda+\ii\eps)=(UF)(\lambda)
\]
uniformly on compacts. With this notation, we have
\begin{equation}
I(F)=
\frac1\pi\lim_{\eps\to0+} \int_a^b \jap{(\Im M_\alpha(\lambda+\ii\eps))(UF)(\lambda+\ii\eps),(UF)(\lambda+\ii\eps)}\dd\lambda. 
\label{X3}
\end{equation}
Furthermore, from the Herglotz--Nevanlinna integral representation~\eqref{eq:intres} for $M_\alpha$ it is easy to see that 
\[
\sup_{0<\eps<1}\int_a^b\norm{\Im M_\alpha(\lambda+\ii\eps)}\dd\lambda<\infty.
\]
Thus, by the dominated convergence, we can replace $(UF)(\lambda+\ii\eps)$ by $(UF)(\lambda)$ in \eqref{X3}. Finally, using the Stieltjes inversion formula for $M_\alpha$, we find
\[
I(F)=
\int_a^b \jap{\dd\Sigma(\lambda)UF(\lambda),UF(\lambda)}.
\]
(Here we have used that $\Sigma(\{a\})=\Sigma(\{b\})=0$.)
On the other hand, coming back to the definition of $I(F)$, by the spectral theorem for self-adjoint operators,
\[
I(F)=\jap{\chi_{(a,b)}(\bH)F,F}.
\]
Putting this together, we get
\[
\jap{\chi_{(a,b)}(\bH)F,F}
=
\int_a^b \jap{\dd\Sigma(\lambda)UF(\lambda),UF(\lambda)}.
\]
Sending $a\to-\infty$ and $b\to\infty$, we obtain
\[
\norm{F}^2=\int_{-\infty}^\infty\jap{\dd\Sigma(\lambda)UF(\lambda),UF(\lambda)},
\]
with the absolute convergence of the integral in the right-hand side. Thus, $U$ is an isometry of $L^2(\bbR_+;\bbC^2)$ onto a~subspace of $L_{\Sigma}^{2}(\bbR;\bbC^{2})$.

\emph{Step 2: $U$ satisfies the intertwining relation \eqref{eq:intertwine}.}
Let $D_0$ be the space of functions $F\in C^\infty(\bbR_+;\bbC^2)$ that vanish for all sufficiently large $x$ and satisfy the boundary condition~\eqref{eq:a2a}. It is not difficult to check that $D_0$ is dense in $\Dom \bH$ in the graph norm of $\bH$.
For any $F\in D_0$ and $\lambda\in\bbR$, integrating by parts, we find
\begin{align*}
U{\bH}F(\lambda)
&=
\int_0^\infty \Phi(x,\lambda)^*(\bH F)(x)\dd x
=
\int_0^\infty (\bH \Phi(x,\lambda))^* F(x)\dd x
\\
&=
\lambda \int_0^\infty \Phi(x,\lambda)^* F(x)\dd x
=
\lambda UF(\lambda);
\end{align*}
here the boundary terms vanish because both $F$ and $\Phi$ satisfy the boundary condition~\eqref{eq:a2a}. Thus, we see that $U$ maps $D_0$ into the domain of $\Lambda$ and 
\[
U\bH F=\Lambda UF
\]
for all $F\in D_0$. Since $D_0$ is dense in $\Dom \bH$ in the graph norm of $\bH$, we obtain that $U\Dom \bH\subset\Dom \Lambda$ and $U\bH=\Lambda U$ holds on $\Dom \bH$, as required. 

\emph{Step 3: the range of $U$ is the whole space $L_{\Sigma}^{2}(\bbR;\bbC^{2})$.}
Observe first that by the intertwining relation \eqref{eq:intertwine}, the range $\Ran U$ and its orthogonal complement $(\Ran U)^\perp$ are invariant subspaces for $\Lambda$. Suppose, to get a contradiction, that there is a non-zero element $G\in L_{\Sigma}^{2}(\bbR;\bbC^{2})$ such that $G\perp \Ran U$. Then, for any bounded function $\xi$,  we also have $\xi(\Lambda)G\perp\Ran U$. It follows that we can find a bounded interval $(a,b)\subset\bbR$ such that $G_0:=\chi_{(a,b)}(\Lambda)G$ is non-zero and $G_0\perp\Ran U$.

Let us first consider the case when the boundary parameter $\alpha\in\bbC$. For $\eps>0$, we make use of the elements $F_{\eps,\uparrow}$ and $F_{\eps,\downarrow}$ of $L^2(\bbR_+;\bbC^2)$ defined by 
\[
F_{\eps,\uparrow}
:=
\frac1\eps
\begin{pmatrix}
\chi_{(0,\eps)}\\ 0
\end{pmatrix},
\quad
F_{\eps,\downarrow}
:=
\frac1\eps
\begin{pmatrix}
0\\\chi_{(0,\eps)}
\end{pmatrix}.
\]
Consider $UF_{\eps,\uparrow}$ and $UF_{\eps,\downarrow}$. By the boundary condition \eqref{eq:phi,theta_bc_cond} and the continuity of $\Phi(x,\lambda)$ in $x$, we find that 
\begin{align*}
\lim_{\eps\to0+}UF_{\eps,\uparrow}(\lambda)&=\frac{1}{\sqrt{1+|\alpha|^{2}}} e_{\uparrow}\, , \quad \mbox{ where }\, e_{\uparrow}:=\begin{pmatrix}1\\0\end{pmatrix},\\
\quad
\lim_{\eps\to0+}UF_{\eps,\downarrow}(\lambda)&=\frac{1}{\sqrt{1+|\alpha|^{2}}} e_{\downarrow}\, , \quad \mbox{ where }\, e_{\downarrow}:=\begin{pmatrix}0\\1\end{pmatrix},
\end{align*}
uniformly over $\lambda\in(a,b)$. Since $G_0\perp \Ran U$, we find that 
\[
G_0\perp \chi_{(a,b)}e_{\uparrow}
\quad \text{ and }\quad
G_0\perp \chi_{(a,b)}e_{\downarrow}.
\]
It follows that 
\[
G_0\perp \xi(\Lambda)\chi_{(a,b)}e_{\uparrow}
\quad \text{ and }\quad
G_0\perp \xi(\Lambda)\chi_{(a,b)}e_{\downarrow}
\]
for any bounded Borel function $\xi$. It follows that $G_0=0$, which is a contradiction. 

Finally, if $\alpha=\infty$, we take 
\[
F_{\eps,\uparrow}
:=
-\frac2{\eps^2}
\begin{pmatrix}
\chi_{(0,\eps)}\\ 0
\end{pmatrix},
\quad
F_{\eps,\downarrow}
:=
-\frac2{\eps^2}
\begin{pmatrix}
0\\\chi_{(0,\eps)}
\end{pmatrix},
\]
and then in a similar way
\[
UF_{\eps,\uparrow}(\lambda)\to e_{\uparrow},
\quad
UF_{\eps,\downarrow}(\lambda)\to e_{\downarrow}
\]
as $\eps\to0+$.
The proof of Proposition~\ref{thm.sp.decomp} is complete.
\qed

\begin{remark*}
The diagonalisation theorems of this kind are usually proved by restricting the operator to a finite interval $(0,b)$ and taking $b\to\infty$; this method goes back to Weyl and Titchmarsh. 
The proof presented here is very close to the ``method of directing functionals'' introduced by M.~G.~Krein in \cite{Krein}. This method is explained in detail in \cite[Appendix II]{AG}. See also a more modern source \cite{FHL}. Our proof of completeness (Step 3) seems slightly more straightforward than the proofs we have seen in this literature. 
\end{remark*}

\end{document}